\documentclass{amsart}
\usepackage{graphicx}
\usepackage{amsmath, amsfonts,amssymb}
\usepackage{eucal}
\usepackage{amsthm}
\usepackage{mathrsfs}
\usepackage{etex}
\usepackage[all,cmtip,arrow,2cell]{xy}
\usepackage[colorlinks,
             linkcolor=black,
             citecolor=blue,
             bookmarksnumbered=true]{hyperref}
\usepackage{appendix}
\usepackage{latexsym, amscd}
\usepackage{float}
\usepackage{soul}
\restylefloat{table}

\newtheorem{prop}{Proposition}[section]
\newtheorem{thm}{Theorem}[section]
\newtheorem{cor}{Corollary}[section]
\newtheorem{lem}[thm]{Lemma}

\theoremstyle{definition}
\newtheorem{dfn}{Definition}[section]
\newtheorem{rmk}{Remark}[section]

\newtheorem{ex}{Example}[section]

\theoremstyle{remark}
\theoremstyle{remark}

\theoremstyle{remark}
\newtheorem{innercustomthm}{Theorem}
\newenvironment{customthm}[1]
  {\renewcommand\theinnercustomthm{#1}\innercustomthm}
  {\endinnercustomthm}

\def\HA{\mathbb{H}}

\def\et{\acute{e}t}

\def\cR{\mathcal R}
\def\cK{\mathcal K}
\def\sD{\mathscr D}

\def\G{\mathcal G}
\def\Ll{\mathcal L}

\def\R{\mathbb{R}}

\def\Mfd{\mathit{Mfd}}
\def\nMfd{n\mbox{-}\Mfd}
\def\Mfdt{\widetilde{\Mfd}}
\def\bl{\mspace{3mu}\cdot\mspace{3mu}}
\def\Bor{\mathfrak{Bor}\left(G \acts M\right)}

\def\cL{\mathcal L}
\def\C{\mathscr C}
\def\X{\mathscr X}
\def\Y{\mathscr Y}
\def\Z{\mathscr Z}

\def\K{\mathscr K}

\def\Fun{\mathbf{Fun}}

\def\cF{\mathcal F}
\def\G{\mathcal G}
\renewcommand\H{\mathcal H}

\def\i{\infty}

\def\n1i{\left(n+1,1\right)}
\def\Gpd{\mathbf{Gpd}}
\def\nGpd{\mathbf{Gpd}_n}
\def\iGpd{\mathbf{Gpd}_\i}

\def\LiGpd{\widehat{\mathbf{Gpd}}_\i}
\def\longlongrightarrow{-\!\!\!-\!\!\!-\!\!\!-\!\!\!-\!\!\!-\!\!\!\longrightarrow}
\def\longlonglongrightarrow{-\!\!\!-\!\!\!-\!\!\!-\!\!\!-\!\!\!-\!\!\!\longlongrightarrow}

\def\longlonglonglongrightarrow{-\!\!\!-\!\!\!-\!\!\!-\!\!\!-\!\!\!-\!\!\!\longlonglongrightarrow}
\def\longlonglonglonglongrightarrow{-\!\!\!-\!\!\!-\!\!\!-\!\!\!-\!\!\!-\!\!\!\longlonglonglongrightarrow}

\def\Top{\mathbb{T}\!\operatorname{op}}

\def\Pshi{\Psh_\i}

\def\icat{\left(\i,1\right)\mbox{-category}}
\def\icat{\mbox{$\i$-category}}

\def\colim{\underrightarrow{\operatorname{colim}\vspace{0.5pt}}\mspace{4mu}}
\renewcommand{\lim}{\varprojlim}

\newcommand{\Hom}{\operatorname{Hom}}
\DeclareMathOperator{\St}{St}
\DeclareMathOperator{\Set}{Set}
\DeclareMathOperator{\Sh}{Sh}
\DeclareMathOperator{\Shi}{Sh_\i}
\def\LShi{\widehat{\Sh}_\i}
\DeclareMathOperator{\Psh}{Psh}

\DeclareMathOperator{\Lan}{Lan}

\newcommand{\Etds}{\mathfrak{EtDiffSt}_\i}
\newcommand{\Etdsd}{\mathfrak{EtDiffSt}}
\newcommand{\Etdsn}{n\mbox{-}\mathfrak{EtDiffSt}_\i}
\newcommand{\LieGpd}{\mathbf{LieGpd}}
\newcommand{\Emb}{\operatorname{Emb}\left(\RR^n\right)}

\def\RR{\mathbb{R}}

\renewcommand{\O}{\mathcal{O}}

\newcommand{\Adj}[4]{\xymatrix@1{#2 \ar@<-0.5ex>[r]_-{#4} & #3 \ar@<-0.5ex>[l]_-{#1}}}

\newcommand{\hocolim}{\operatorname{hocolim}}

\def\rrrarrow{\hspace{.05cm}\mbox{\,\put(0,-3){$\rightarrow$}\put(0,1){$\rightarrow$}\put(0,5){$\rightarrow$}\hspace{.45cm}}}
\def\rrarrow{  \hspace{.05cm}\mbox{\,\put(0,-2){$\rightarrow$}\put(0,2){$\rightarrow$}\hspace{.45cm}}}
\def\acts{\hspace{.1cm}{\setlength{\unitlength}{.30mm}\linethickness{.09mm}
                        \begin{picture}(8,8)(0,0)\qbezier(7,6)(4.5,8.3)(2,7)\qbezier(2,7)(-1.5,4)(2,1)\qbezier(2,1)(4.5,-.3)(7,2)
                                                 \qbezier(7,6)(6.1,7.5)(6.8,9)\qbezier(7,6)(5,6.1)(4.2,4.4)
                        \end{picture}\hspace{.1cm}}}

\def\acted{\hspace{.1cm}{\setlength{\unitlength}{.30mm}\linethickness{.09mm}
                        \begin{picture}(8,8)(0,0)\qbezier(1,6)(3.5,8.3)(6,7)\qbezier(6,7)(9.5,4)(6,1)\qbezier(6,1)(3.5,-.3)(1,2)
                                                 \qbezier(1,6)(1.9,7.5)(1.2,9)\qbezier(1,6)(3,6.1)(3.8,4.4)
                        \end{picture}\hspace{.1cm}}}

\def\longlongrightarrow{-\!\!\!-\!\!\!-\!\!\!-\!\!\!-\!\!\!-\!\!\!\longrightarrow}
\def\longlonglongrightarrow{-\!\!\!-\!\!\!-\!\!\!-\!\!\!-\!\!\!-\!\!\!-\!\!\!-\!\!\!\longrightarrow}
\DeclareMathOperator{\sk}{sk}
\begin{document}
\title[\resizebox{4.5in}{!}{On The Homotopy Type of Higher Orbifolds and Haefliger Classifying Spaces}]{On The Homotopy Type of Higher Orbifolds and Haefliger Classifying Spaces}
\author{David Carchedi}

\maketitle

\begin{abstract}
We describe various equivalent ways of associating to an orbifold, or more generally a higher \'etale differentiable stack, a weak homotopy type. Some of these ways extend to arbitrary higher stacks on the site of smooth manifolds, and we show that for a differentiable stack $\X$ arising from a Lie groupoid $\G,$ the weak homotopy type of $\X$ agrees with that of $B\G.$ Using this machinery, we are able to find new presentations for the weak homotopy type of certain classifying spaces. In particular, we give a new presentation for the Borel construction $M \times_G EG$ of an almost free action of a Lie group $G$ on a smooth manifold $M$ as the classifying space of a category whose objects consists of smooth maps $\RR^n \to M$ which are transverse to all the $G$-orbits, where $n=\dim M - \dim G.$ We also prove a generalization of Segal's theorem, which presents the weak homotopy type of Haefliger's groupoid $\Gamma^q$ as the classifying space of the monoid of self-embeddings of $\RR^q,$ $B\left(\operatorname{Emb}\left(\RR^q\right)\right),$ and our generalization gives analogous presentations for the weak homotopy type of the Lie groupoids $\Gamma^{Sp}_{2q}$ and $\cR\Gamma^q$ which are related to the classification of foliations with transverse symplectic forms and transverse metrics respectively. We also give a short and simple proof of Segal's original theorem using our machinery.
\end{abstract}

\tableofcontents

\section{Introduction}

In this paper, we explain how to functorially associated to an orbifold, or more generally to a higher \'etale differentiable stack, a canonical weak homotopy type. This builds on work of Dugger and Schreiber, who address this question for arbitrary higher stacks in \cite{dugger} and \cite{dcct}, and also of Moerdijk who addresses this question in terms of \'etale Lie groupoids in \cite{Weak}. When restricting to the case of \'etale stacks, we are aided by our relatively recent categorical characterization of \'etale differentiable stacks and their higher categorical analogues \cite{prol,higherme}. Using this characterization, we are able to directly express the weak homotopy type of an $n$-dimensional higher \'etale differentiable stack as a homotopy colimit of a diagram of spaces indexed by the monoid of smooth embeddings of $\RR^n,$ $\Emb$. This is intimately linked with Segal's celebrated theorem expressing the weak homotopy type of the classifying space $B\Gamma^n$ of Haefliger's groupoid $\Gamma^n$ as the classifying space $B\left(\Emb\right)$ \cite{Segal}, and we are able to exploit this connection to find explicit presentations for the homotopy type of various examples of \'etale differentiable stacks as the classifying space of certain categories related to the geometry of the stack in question. For example, we are able to express the homotopy type of the quotient stack $M//G$ of a smooth manifold by an almost free Lie group action as the classifying space of a category whose objects consist of smooth maps $\RR^n \to M$ which are transverse to all the $G$-orbits, where $n=\dim M - \dim G.$ (Recall that an action of a Lie group is \emph{almost free} if each stabilizer group $G_x$ is discrete.) Since the weak homotopy type of $M//G$ is the Borel construction $M \times_G EG,$ this gives a new presentation for the weak homotopy type of the Borel construction.

\subsection{Higher \'Etale Differentiable Stacks}
\'Etale differentiable stacks model quotients of smooth manifolds by certain symmetries, and their points can possess intrinsic (discrete) automorphism groups. The most prevalently studied subclass of \'etale stacks is that of smooth orbifolds, however there are other important geometric examples of \'etale differentiable stacks not of this form, e.g. quotients of smooth manifolds by almost free actions of non-compact Lie groups, and leaf spaces of foliated manifolds. \'Etale differentiable stacks have been studied by various authors, c.f. \cite{Ie,Dorette,stacklie,hepworth,morsifold,Wockel,Giorgio,gerbes,etalspme}. \emph{Higher} \'etale differentiable stacks are higher categorical analogues of \'etale differentiable stacks allowing not only points to have automorphisms, but also the automorphisms of points to have automorphisms themselves, and so on. 

One common geometric presentation for \'etale differentiable stacks is the bicategory whose objects are \'etale Lie groupoids, and whose morphisms are given by groupoid principal bundles (aka Hilsum-Skandalis maps), and whose $2$-morphisms are morphisms of such bundles. More generally, higher \'etale differentiable stacks are precisely those higher stacks arising from \'etale simplicial manifolds. Recently, we have given a complete categorical characterization of higher \'etale differentiable stacks which is independent of the description in terms of simplicial manifolds (see \cite{prol} for the $2$-categorical case, and \cite{higherme} for the general case). The characterization is relatively simple. Let $\nMfd^{\et}$ denote the category of smooth $n$-manifolds and their local diffeomorphisms. There is a canonical faithful functor $$j_n:\nMfd^{\et} \to \Mfd$$ to the category of all smooth manifolds and their smooth maps, and it induces a functor $$j^*_n:\Shi\left(\Mfd\right) \to \Shi\left(\nMfd^{\et}\right)$$ between their associated $\i$-categories of higher stacks by restriction along $j_n.$ This functor has a left adjoint $$j^n_!:\Shi\left(\nMfd^{\et}\right) \to \Shi\left(\Mfd\right)$$ called the \emph{$n$-dimensional \'etale prolongation functor}. A higher stack $\X$ on the category of smooth manifolds $\Mfd$ is an $n$-dimensional higher \'etale differentiable stack if and only if it is in the essential image of $j^n_!$. (In fact, $\Shi\left(\nMfd^{\et}\right)$  is equivalent to the $\i$-category of $n$-dimensional higher \'etale differentiable stacks and their local diffeomorphisms.)

\subsection{Foliations with Transverse Structures}
In his seminal paper \cite{Haefliger}, Haefliger introduces several Lie groupoids whose classifying spaces are related to the classification of foliations with certain transverse structures. E.g., he constructs a Lie groupoid $\Gamma^{Sp}_{2q}$ whose classifying space $B\Gamma^{Sp}_{2q}$ is related to the classification of integrable homotopy classes of foliations with transverse symplectic structure, and he constructs a Lie groupoid $\cR\Gamma^q$ corresponding to the case of transverse Riemannian metrics. See Section \ref{sec:foliations} for more details (in particular Theorem \ref{thm:Segaltrans}).

In \cite{prol}, we show that for any manifold $M,$ submersions $$M \to \left[\Gamma^{Sp}_{2q}\right],$$ where $\left[\Gamma^{Sp}_{2q}\right]$ is the differentiable stack associated to $\Gamma^{Sp}_{2q},$ are in bijection with foliations of $M$ of codimension $2q$ equipped with a transverse symplectic structure, and similarly submersions $$M \to \left[\cR\Gamma^q\right]$$ are the same as foliations with transverse metrics. We also show in op. cit. that $\left[\Gamma^{Sp}_{2q}\right]=j^{2q}_!\left(\mathcal{S}_{2q}\right),$ where $\mathcal{S}_{2q}$ is the sheaf on $2q$-manifolds and their local diffeomorphisms which assigns to a $2q$-manifold its set of symplectic forms, and similarly $\left[\cR\Gamma^q\right]=j^{q}_!\left(\mathcal{R}_q\right)$ where $\mathcal{R}_q$ is the sheaf on $q$-manifolds and their local diffeomorphisms which assigns to each $q$-manifold its set of Riemannian metrics.

By Corollary \ref{cor:diffhom} of this paper, the homotopy type of $\left[\Gamma^{Sp}_{2q}\right]$ is $B\Gamma^{Sp}_{2q}$ and the homotopy type of $\left[\cR\Gamma^q\right]$ is $B\cR\Gamma^q.$ Since, the stacks $\left[\Gamma^{Sp}_{2q}\right]$ and $\left[\cR\Gamma^q\right]$ are determined by the sheaves $\mathcal{S}_{2q}$ and $\mathcal{R}_q$ respectively, there must be a natural way of expressing the homotopy type of $B\Gamma^{Sp}_{2q}$ and $B\cR\Gamma^q$ in terms of these sheaves. Spelling this out leads to generalizations of Segal's theorem, namely Theorems \ref{thm:segalsymp} and \ref{thm:segalriem} of this article.

\subsection{Organization and Main Results}
In Section \ref{sec:higher orbifolds} we review the theory of \'etale differentiable stacks and their higher categorical generalizations. In particular, we review some of the main results of Chapter 6.1 of \cite{higherme}.

In Section \ref{sec:homotopy type}, we explain various ways of associating to a higher \'etale differentiable stack a weak homotopy type. In particular, in Section \ref{sec:funigpd}, we introduce the \emph{fundamental $\i$-groupoid} functor $$\Pi_\i:\Shi\left(\Mfd\right) \to \iGpd,$$ which is a functor associating to any higher stack $\X$ on the site of smooth manifolds its weak homotopy type (regarded as an $\i$-groupoid). We then prove the following proposition, which is a strengthening of results found in \cite{dugger} and \cite{dcct}:

\begin{prop}\label{prop:pi2}
The fundamental $\i$-groupoid functor $$\Pi_\i:\Shi\left(\Mfd\right) \to \iGpd$$ sends every manifold $M$ in $\widetilde{\Mfd}$ to its underlying homotopy type, where $\widetilde{\Mfd}$ is the category of all topological spaces with a smooth atlas, i.e. smooth manifolds which need not be Hausdorff or $2^{nd}$-countable.
\end{prop}

This turns out to be an essential strengthening as there are many important examples of Lie groupoids whose arrow spaces are non-Hausdorff, e.g. Haefliger's groupoid $\Gamma^q.$ 

In Section \ref{sec:fat}, we explain the relationship between the fat geometric realization functor $$|| \mspace{3mu} \bullet \mspace{3mu}||:\Top^{\Delta^{op}} \to \Top$$ and the fundamental $\i$-groupoid functor
$$\Pi_\i:\Shi\left(\Mfd\right) \to \iGpd.$$ Namely, we prove the following theorem:

\begin{thm}
Suppose that $X_\bullet:\Delta^{op} \to \widetilde{\Mfd}$ is a simplicial manifold and let $\left[X_\bullet\right]$ be its associated higher stack. Then $\Pi_\i \left[X_\bullet\right]$ has the same weak homotopy type as $||X||.$
\end{thm}

As an immediate corollary we get:

\begin{cor}
For $\left[\G\right]$ a differentiable stack, $\Pi_\i\left(\left[\G\right]\right)$ has the same weak homotopy type as $B\G.$
\end{cor}

Along the way, we prove two useful facts about fat geometric realization:

\begin{lem}
If $f:X_\bullet \to Y_\bullet$ is a map of (semi-)simplicial spaces which is a degree-wise weak homotopy equivalence, then the induced map $$||X|| \to ||Y||$$ is a weak homotopy equivalence.
\end{lem}

\begin{lem}
If $X_\bullet:\Delta^{op} \to \Top$ is any simplicial space then $||X||$ is the homotopy colimit of $X_\bullet.$
\end{lem}

The above two lemmas appear to be well-known in some crowds, however we found it a worthy exercise to write up their full proofs as there are some subtleties to deal with when the spaces involved are not $T1.$

In Section \ref{sec:monoid1}, we introduce the discrete monoid of self-embeddings of $\R^n$, $\Emb,$ and in Section \ref{sec:main} we use this monoid in an essential way to give an alternative description of the weak homotopy type $\Pi_\i\left(\X\right)$ of a higher stack $\X,$ when $\X$ is a higher \'etale differentiable stack.

Our main theorem is the following:

\begin{thm}
Suppose that $F$ in $\Shi\left(\nMfd^{\et}\right)$ is an $\i$-sheaf on $n$-manifolds and their local diffeomorphisms, and let $\X=j^n_!F$ be its associated $n$-dimensional \'etale differentiable $\i$-stack. Then the $\i$-groupoid $\Pi_\i\left(\X\right)$ can be expressed as the colimit of the following composite:

$$\Emb^{op} \to \left(\nMfd^{\et}\right)^{op} \stackrel{F}{\longrightarrow} \iGpd.$$
\end{thm}

Using this theorem, we give an easy proof of the following celebrated theorem of Segal's:

\begin{thm} (Proposition 1.3 of \cite{Segal})\\
Let $\Gamma^n$ be the $n$-dimensional Haefliger groupoid of Example \ref{ex:Haefliger}. Then there is a weak homotopy equivalence between classifying spaces:
$$B\left(\Emb\right) \to B \Gamma^n.$$ 
\end{thm}

Similar arguments lead to generalizations of Segal's theorem:

\begin{thm}\label{thm:segalsymp}
Let $Sp_{2n}$ denote the following category. The objects consist of symplectic forms $\omega$ on $\RR^{2n}$. An arrow $\omega \to \omega'$ between two such symplectic forms is an embedding $$\varphi:\R^{2n} \hookrightarrow \R^{2n}$$ such that $\varphi^*\omega'=\omega.$ Then there is a weak homotopy equivalence between classifying spaces:
$$B\left(Sp_{2n}\right) \to B \Gamma^{Sp}_{2n},$$ where $\Gamma^{Sp}_{2n}$ is the Lie groupoid from Example \ref{ex:sympgpd}.
\end{thm}

\begin{thm}
Let $\cR \mbox{iem}_n$ denote the following category. The objects consist of Riemannian metrics $g$ on $\RR^{n}$. An arrow $g \to g'$ between two such metrics is an embedding $$\varphi:\R^{n} \hookrightarrow \R^{n}$$ such that $\varphi^*g'=g.$ Then there is a weak homotopy equivalence between classifying spaces:
$$B\left(\cR \mbox{iem}_n\right) \to B \cR\Gamma^{n},$$ where $\cR\Gamma^{n}$ is the Lie groupoid from Example \ref{ex:riemgpd}.
\end{thm}

Finally, in Section \ref{sec: Borel}, we derive an alternative description of the Borel construction $M \times_G EG$ of an almost free action of a Lie group $G$ on a smooth manifold $M.$ Firstly, if $n=\dim M - \dim G$, then we define the \emph{Borel category} $\Bor$. Its objects are smooth maps $$f:\R^n \to M$$ which are transverse to all the $G$-orbits, and its arrows from $f$ to $f'$ are pairs $\left(\varphi,\tau\right)$ with $$\varphi:\R^n \hookrightarrow \R^n$$ an embedding, and with $$\tau:\R^n \to G$$ a smooth map such that for all $x \in \R^n$ we have $$f\left(x\right)=\tau\left(x\right) \cdot f'\left(\varphi\left(x\right)\right).$$ We then prove the following theorem:

\begin{thm}
Let $G$ be a Lie group of dimension $k$ and let $M$ be a smooth $\left(n+k\right)$-manifold equipped with an almost free $G$-action. Then $M \times_G EG$ has the same weak homotopy type as $B\left(\Bor\right)$.
\end{thm}

\subsection{Higher Categorical Background}
Higher category theory, and in particular the language of $\i$-categories, will be used throughout this paper. However, the reader need not be familiar with this language in order to read this paper, as long as they comfortable with some basic notions from category theory and are willing to accept at faith certain generalizations of these concepts to higher categories. For the reader's convenience, we give a brief introduction to this language in Appendix \ref{sec:infinity}, which we hope to be sufficient for the purposes of reading this paper. The reader will be assumed to have some familiarity with basics of homotopy theory.

\subsection*{Acknowledgments}
We are grateful to Baptiste Devyver, Achim Krause, Tyler Lawson, Urs Schreiber, and Danny Stevenson for helpful discussions. We are also grateful for the pleasant working environments provided by both the mathematics department of the University of British Columbia, as well as the Max Planck Institute for Mathematics, where a large portion of this research was conducted.

\section{A Review of Higher Orbifolds and \'Etale Stacks}\label{sec:higher orbifolds}
In this section, we recall for the reader some of the theory of \'etale differentiable stacks and their higher categorical analogues. 

\subsection{Lie Groupoids}

\begin{dfn}
Denote by $\Mfd$ the category of all smooth manifolds, where we assume that each manifold $M$ in $\Mfd$ is $2^{nd}$ countable and Hausdorff so that $\Mfd$ is essentially small by Whitney's embedding theorem. We denote by $\Mfdt$ the category of all smooth manifolds without these conditions, i.e. the category of topological spaces with a smooth atlas.
\end{dfn}

\begin{dfn}
An \textbf{Lie groupoid} is a groupoid object in the category $\Mfdt$ such that the source and target maps are submersions. Explicitly, it is a diagram

$$\xymatrix{ {\G_1 \times _{\G_0}\G_1} \ar[r]^(0.6){m} & \G_1 \ar@<+.7ex>^s[r]\ar@<-.7ex>_t[r] \ar@(ur,ul)[]_{\hat i}  & \G_0 \ar@/^1.65pc/^{1} [l] }$$
of smooth manifolds and smooth maps between them satisfying the usual axioms. Forgetting the manifold structure, one obtains an ordinary small groupoid. Let $s$ and $t$ denote the source and target maps of a groupoid respectively. Such a Lie groupoid $\G$ is \textbf{\'etale} if the source map $s$ (and therefore the target map $t$) is a local diffeomorphism.

Lie groupoids form a $2$-category with smooth functors as $1$-morphisms and smooth natural transformations as $2$-morphisms. We will denote this $2$-category by $\LieGpd$.
\end{dfn}

\begin{rmk}
Often in other literature, Lie groupoids are required to have a Hausdorff object space. However, we will only be concerned with the differentiable stacks associated to Lie groupoids, and every Lie groupoid $\G$ in the sense we defined can be replaced with another Lie groupoid $\G'$ in the more restrictive sense, such that they have the same associated stack.
\end{rmk}

Given a Lie groupoid $\G,$ and a point $x \in \G_0,$ one may consider the isotropy group $\G_x$ of $x$, which is the collection of all arrows $g:x \to x$ of $\G,$ considered as a closed submanifold of $\G_1.$ $\G_x$ naturally has the structure of a Lie group.

\begin{dfn}
A \textbf{foliation groupoid} is a Lie groupoid $\G$ such that for each $x \in \G_0,$ the isotropy group $\G_x$ is discrete.
\end{dfn}

\begin{rmk}
Every \'etale Lie groupoid is in particular a foliation groupoid.
\end{rmk}

\begin{rmk}
A Lie groupoid $\G$ is a foliation groupoid if and only if the orbits of $\G$ are the leaves of a regular foliation of $\G_0$, that is, all the orbits of $\G$ have the same dimension.
\end{rmk}

\begin{dfn}
A Lie groupoid $\G$ is \textbf{proper} if the map $$\left(s,t\right):\G_1 \to \G_0 \times \G_0$$ is proper. (Recall that a continuous map is proper if the pre-image of any compact set is compact).
\end{dfn}


\begin{dfn}\label{dfn:Morita}
A smooth functor $\varphi: \H \to \G$ of Lie groupoids is a \textbf{Morita equivalence} if the following two properties hold:

\begin{itemize}
\item[i)] (Essentially Surjective)\\
The map $t \circ pr_1: \G_1 \times_{\G_0} \H_0 \to \G_0$ is a surjective submersion, where $\G_1 \times_{\G_0} \H_0$ is the fibered product

$$\xymatrix{\G_1 \times_{\G_0} \H_0 \ar[r]^-{pr_2} \ar[d]_-{pr_1} & \H_0 \ar[d]^-{\varphi} \\
\G_1 \ar[r]^-{s} & \G_0.}$$
\item[i)] (Fully Faithful)
The following is a fibered product:

$$\xymatrix{\H_1 \ar[r]^-{\varphi} \ar[d]_-{\left(s,t\right)} &\G_1 \ar[d]^-{\left(s,t\right)}  \\
\H_0 \times \H_0 \ar[r]^-{\varphi \times \varphi} & \G_0 \times \G_0.}$$
\end{itemize}
Two Lie groupoids $\Ll$ and $\K$ are \textbf{Morita equivalent} if there is a chain of Morita equivalences $\Ll \leftarrow \H \rightarrow \K$.
\end{dfn}


\begin{prop}\label{prop:folgpd} (Proposition 5.20 of \cite{fol})\\
A Lie groupoid $\G$ is a foliation groupoid if and only if it is Morita equivalent to an \'etale Lie groupoid.
\end{prop}

\begin{proof}
For the full proof, see \cite{fol}, p. 136. Let $T \hookrightarrow \G_0$ be a complete transversal to the foliation on $\G_0$ induced by the orbits of the groupoid, that is $T$ is a closed submanifold which hits every orbit at least once. Then $\G_T$ is \'etale and the canonical map $\G_T \to \G$ is a Morita equivalence.
\end{proof}

\begin{rmk}
If $\G$ is a foliation groupoid and the orbits of $\G$ all have codimension $q,$ then the \'etale Lie groupoid $\G_T$ constructed above has dimension $q$.
\end{rmk}

There are many interesting examples of Lie groupoids important for this paper:

\begin{ex}\label{ex:id}
Any manifold $M$ can be considered as a Lie groupoid $M^{id}$, whose only arrows are the identity arrows. $M^{id}$ is trivially an \'etale Lie groupoid.
\end{ex}

\begin{ex}
If $G$ is a Lie group, then $G$ may be considered as a Lie groupoid $$G \rrarrow *$$ whose space of objects is the one point space.
\end{ex}

\begin{ex}\label{ex:pair}
If $f:N \to M$ is a submersion, then we can form the \textbf{pair groupoid} $Pair\left(f\right)$ whose object space is $N$ and whose arrow space is $N \times_M N$. Here, an element $\left(x,y\right) \in N \times_M N$ is viewed as an arrow from $y$ to $x$ and composition is given by the rule $$\left(x,y\right) \circ \left(y,z\right)=\left(x,z\right).$$ There is a canonical smooth functor $Pair\left(f\right) \to M^{id}$ which is a Morita equivalence.
\end{ex}

\begin{ex}\label{ex:cech}
If $\mathcal{U}=\left(U_\alpha \hookrightarrow M\right)$ is an open cover of a manifold $M,$ then associated to this cover there is the \textbf{C\v{e}ch groupoid} $M_\mathcal{U}$. This groupoid is the pair groupoid of the projection $$\coprod_\alpha U_\alpha \to M,$$ so the object space of this Lie groupoid is the disjoint union $$\coprod_\alpha U_\alpha$$ and the arrows space is the disjoint union $$\coprod_{\alpha,\beta} U_\alpha \cap U_\beta.$$ There is a canonical smooth functor $M_\mathcal{U} \to M^{id}$ which is a Morita equivalence. All C\v{e}ch groupoids are \'etale Lie groupoids.
\end{ex}

\begin{rmk}
Example \ref{ex:id} is the special case of Example \ref{ex:cech} where the open cover is the identity map.
\end{rmk}

\begin{ex}\label{ex:action}
Let $G$ be a Lie group acting on a smooth manifold $M$. Associated to this action is the \textbf{action groupoid} $G \ltimes M$. The object space of this groupoid is $M$ and the arrow space is the manifold $G \times M,$ where a pair $$\left(g,x\right) \in G \times M$$ is an arrow $$x \to g \cdot x,$$ with the obvious composition rule. If the action is almost free (i.e. each stabilizer group $G_x$ is discrete), then $G \ltimes M$ is a foliation groupoid. If $G$ itself is discrete, then $G \ltimes M$ is an \'etale Lie groupoid. If the action of $G$ is proper, then the action groupoid is proper.
\end{ex}

\begin{ex}\label{ex:Haefliger}
Let $M$ be a manifold. Consider the sheaf $\cL$ of sets on $M$ which assigns to each open subset $U$ of $M$ the set of (abstract) local diffeomorphisms of $U$ into $M$. (This is the sheafification of the presheaf of embeddings). Let $$s:\Gamma\left(M\right)_1 \to M$$ be the \'etal\'e space of this sheaf, i.e. $s$ is the local homeomorphism over $M$ whose sheaf of sections is $\cL.$ Since $s$ is a local homeomorphism, the space $\Gamma\left(M\right)_1$ inherits a smooth atlas, making it a smooth manifold (although it is highly non-Hausdorff). As a set, $\Gamma\left(M\right)_1$ is the disjoint union of the stalks of $\cL$:
$$\Gamma\left(M\right)_1\cong \coprod_{x \in M} \cL_x.$$ A point in $\cL_x$ is of the form $\mathit{germ}_x\varphi$ for $\varphi:U \to M$ a local diffeomorphism from a neighborhood $U$ of $x.$ The assignment $$\mathit{germ}_x\varphi \mapsto \varphi\left(x\right)$$ is well defined and assembles into a local diffeomorphism $$t:\Gamma\left(M\right)_1 \to M.$$ Composition of germs gives a groupoid structure, defining an \'etale Lie groupoid $\Gamma\left(M\right),$ called the \textbf{Haefliger groupoid} of $M,$ whose space of objects is $M$ and whose arrows $x \to y$ are given by germs of locally defined diffeomorphisms around $x$ that carry $x$ to $y$. The Haefliger groupoid of $\R^n$ will be denoted by $\Gamma^n$.
\end{ex}

\begin{ex}\label{ex:sympgpd}
Consider $\R^{2n}$ with its canonical symplectic structure, which in coordinates $\left(p_1,\ldots,p_n,q_1,\ldots,q_n\right)$ is given by $$\omega_{can.}=\sum_{i=1}^{n} dp_i \wedge dq_i.$$
Each open subset $U$ is canonically a symplectic manifold with symplectic form $\omega_{can.}|_U,$ and one can consider the presheaf on $\R^{2n}$ assigning to $U$ the set of symplectic embeddings $$U \hookrightarrow \R^{2n},$$ both endowed with their canonical symplectic structure. Its sheafification is the sheaf $\mathcal{LS}$ assigning $U$ the set of local symplectomorphisms $$U \to \R^{2n}.$$ A completely analogous construction to Example \ref{ex:Haefliger} now produces an \'etale Lie groupoid $\Gamma^{Sp}_{2n}$ whose space of objects is $\R^{2n}$ and whose arrows $x \to y$ are given by germs of locally defined symplectomorphisms carrying $x$ to $y$.
\end{ex}

\begin{ex}\label{ex:riemgpd}
Let $\cR$ be the sheaf on $\R^n$ which assigns to each open subset $U$ the set of Riemannian metrics on $U$. Let $N \to \R^n$ be the \'etal\'e space of this sheaf. Then $N$ inherits the structure of a (non-Hausdorff) smooth $n$-dimensional manifold and carries a canonical Riemannian metric. An analogous construction to Example \ref{ex:Haefliger} and Example \ref{ex:sympgpd} yields an \'etale Lie groupoid $\cR\Gamma^n$ whose space of objects is $N,$ and whose arrows $x \to y$ are given by germs of locally defined isometries of $N$ carrying $x$ to $y$. We will denote this groupoid by $\cR\Gamma^n.$
\end{ex}

\subsection{Classifying Spaces of Lie Groupoids and the Classification of Foliations}\label{sec:foliations}

Associated naturally to each Lie groupoid $\G$ is a simplicial manifold $$N\left(\G\right)_\bullet:\Delta^{op} \to \widetilde{Mfd},$$ which is its simplicial nerve. The underlying set is the ordinary nerve of the underlying (discrete) groupoid of $\G,$ so it suffices to describe the manifold structure on each set $N\left(\G\right)_n.$ For $n=0$ $$N\left(\G\right)_0=\G_0$$ and for $n=1$ $$N\left(\G\right)_1=\G_1,$$ which are manifolds, and for $n>1$ $$N\left(\G\right)_n=\G_1 \times_{\G_0} \G_1 \times_{\G_0} \ldots \times_{\G_0} \G_1,$$ which is an iterated pullback of submersions, hence also a manifold.

\begin{dfn}
If $\G$ is a Lie groupoid, we denote by $B\G$ the fat geometric realization (Definition \ref{dfn:fat}) $||N\left(\G\right)||$ of the underlying simplicial space of the simplicial manifold $N\left(\G\right)_\bullet$, and call it the \textbf{classifying space} of $\G$.
\end{dfn}

\begin{ex}
Regarding a Lie group $G$ as a Lie groupoid with one object, its classifying space $BG$ as above is a particular model for the classifying space for principal $G$-bundles. If $G \ltimes G$ is the action groupoid associated to the canonical (left) action of $G$ on itself, the canonical map $$B\left(G \ltimes G\right) \to BG$$ is a universal principal $G$-bundle, that is, $B\left(G \ltimes G\right)$ is an $EG$. More generally, by Proposition 4.3.32 of \cite{dcct}, if $G \acts M$ is a smooth action of a Lie group, then $B\left(G \ltimes M\right)$ is homeomorphic to the Borel construction $M \times_{G} EG,$ where $EG$ is taken to be $B\left(G \ltimes G\right)$ as above.
\end{ex}

The classifying spaces of the Lie groupoids from Examples \ref{ex:Haefliger}, \ref{ex:sympgpd}, and \ref{ex:riemgpd} play an important role in foliation theory. We will summarize an important result of Haefliger's on the classification of regular foliations:

\begin{dfn}\label{dfn:integrably}
Two foliations $\cF_0$ and $\cF_1$ on a manifold $M$ are \textbf{integrably homotopic} if there is a foliation $\mathcal{H}$ on $M \times I$ transverse to each slice $M \times \left\{t\right\}$ and inducing $\cF_i$ on each $M \times \left\{i\right\},$ for $i=0,1.$
\end{dfn}

There exists a smooth functor $$\Gamma^q \to \operatorname{GL}_q$$ between the Haefliger groupoid of $\RR^q$ and the Lie group $GL_q$ regarded as a one-object Lie groupoid, induced by the assignment to each local diffeomorphism $\varphi$ its derivative. By functorially, there is an induced map between classifying spaces $$v:B\Gamma^q \to B\operatorname{GL}_q.$$

The following is Haefliger's theorem:
\begin{thm}(\cite{Haefliger}, Section 3)
If $M$ is a paracompact open manifold of dimension $n,$ then there is a bijection between integrable homotopy classes of foliations of $M$ of codimension $q$ and homotopy classes of lifts
$$\xymatrix@C=2.5cm{& B\Gamma^q \times B\operatorname{GL}_{n-q} \ar[d]^-{v \times id}\\
& B\operatorname{GL}_q \times B\operatorname{GL}_{n-q} \ar[d]^-{\oplus}\\
M \ar[r]^-{\tau M} \ar@{-->}[ruu] & B\operatorname{GL}_n,}$$ where $\tau M$ is the map classifying the tangent bundle of $M$ and $\oplus$ is the Whitney sum map.
\end{thm}

\begin{dfn}
A \textbf{foliation with transverse metric} on $M$ is a foliation $\cF$ on $M$ together with a positive $g \in \Gamma\left(M,\operatorname{Sym}^2T^*M\right)$ such that
\begin{itemize}
\item[a)] For all $x \in M,$ $\ker\left(g_x\right)=\cF_x,$
\item[b)] $g$ is invariant under the flows tangent to the leaves.
\end{itemize}
\end{dfn}

\begin{dfn}
A \textbf{foliation with transverse symplectic structure} on $M$ is a foliation $\cF$ together with a closed $2$-form $\omega$ such that for all $x,$ $\ker\left(\omega_x\right)=\cF_x.$
\end{dfn}

\begin{rmk}
The invariance under flows along leaves in the above definition is in fact automatic by Cartan's magic formula $$\cL_X=\iota_X \circ d + d \circ \iota_X,$$ since $\omega$ is a closed form.
\end{rmk}

Since both $\Gamma^{Sp}_{2q}$ and $\cR\Gamma^q$ are groupoids of germs of local diffeomorphisms, there are analogously defined smooth functors $v$ to $\operatorname{GL}_{2q}$ and $\operatorname{GL_q}$ respectively. Moreover, Definition \ref{dfn:integrably} has an obvious generalization for foliations with transverse metrics and transverse symplectic structures, respectively; one has to simply ask for the homotopy $\mathcal{H}$ to be a foliation of the same type. There is a variant of Haefliger's theorem which also holds in these cases, namely:

\begin{thm}(\cite{Haefliger}, Section 3)\label{thm:Segaltrans}
Let $M$ be a paracompact open manifold of dimension $n.$ There is a bijection between integrable homotopy classes of foliations with transverse metrics on $M$ of codimension $q$ and homotopy classes of lifts
$$\xymatrix@C=2.5cm{& B\cR\Gamma^q \times B\operatorname{GL}_{n-q} \ar[d]^-{v \times id}\\
& B\operatorname{GL}_q \times B\operatorname{GL}_{n-q} \ar[d]^-{\oplus}\\
M \ar[r]^-{\tau M} \ar@{-->}[ruu] & B\operatorname{GL}_n,}$$ and there is also a bijection between integrable homotopy classes of foliations with transverse symplectic forms on $M$ of codimension $2q$ and homotopy classes of lifts
$$\xymatrix@C=2.5cm{& B\Gamma^{Sp}_{2q} \times B\operatorname{GL}_{n-q} \ar[d]^-{v \times id}\\
& B\operatorname{GL}_q \times B\operatorname{GL}_{n-q} \ar[d]^-{\oplus}\\
M \ar[r]^-{\tau M} \ar@{-->}[ruu] & B\operatorname{GL}_n.}$$
\end{thm}

\begin{rmk}
The Lie groupoids $\Gamma^{Sp}_{2q}$ and $\cR\Gamma^q$ are only examples that seem particularly illustrative because they are geometric in nature; Haefliger's theorem in fact applies to a large class \'etale Lie groupoids $\Gamma,$ and for each such $\Gamma$ there is an associated notion of a $\Gamma$-foliation. For $\Gamma^{Sp}_{2q}$ and $\cR\Gamma^q,$ these agree with foliations with transverse symplectic structure and transverse metric respectively.
\end{rmk} 

We end this subsection by citing a celebrated theorem of Segal:

\begin{thm} (Proposition 1.3 of \cite{Segal})\\
Let $\Emb$ denote the discrete monoid of self embeddings of $\RR^n$. Then there is a weak homotopy equivalence between classifying spaces:
$$B\left(\Emb\right) \to B \Gamma^n.$$ 
\end{thm}

One of the main results of this paper is a generalization of this result  for the Lie groupoids $\Gamma^{Sp}_{2q}$ and $\cR\Gamma^q.$ In fact, it applies to to the much larger class of \'etale Lie groupoids for which Haefliger's theorem holds.

\subsection{Higher Differentiable Stacks}\label{sec:diffstacks}

In this subsection, we will give a streamlined review of the theory of higher stacks and \'etale differentiable stacks needed for this article. There is a more detailed and precise account given in Appendix \ref{sec:stacks}.

\subsection{Differentiable Stacks}

\begin{dfn}
Denote by $\Psh_1\left(\Mfd\right)$ the bicategory of (possibly weak) $2$-functors $$\X:\Mfd^{op} \to \Gpd$$ from (the opposite of) the category $\Mfd$ into the bicategory of (essentially small) groupoids.
\end{dfn}

Let $\G$ be a Lie groupoid. Then $\G$ determines a weak presheaf of groupoids on $\Mfd$ by the rule

\begin{equation*}
M \mapsto \Hom_{\LieGpd}\left(M^{id},\G\right),
\end{equation*}
This defines a $2$-functor $\tilde y: \LieGpd \to \Psh_1\left(\Mfd\right)$ and we have the obvious commutative diagram

$$\xymatrix{\Mfd  \ar[d]_{\left(\mspace{2mu} \cdot \mspace{2mu}\right)^{(id)}} \ar[r]^{y} & \Psh\left(\Mfd\right) \ar^{\left(\mspace{2mu} \cdot \mspace{2mu}\right)^{(id)}}[d]\\
\LieGpd \ar_{\tilde y}[r] & \Psh_1\left(\Mfd\right),}$$
where $y$ denotes the Yoneda embedding.

\begin{dfn}
A weak presheaf of groupoids $\X:\Mfd^{op} \to \Gpd$ is a \textbf{stack} (or $1$-sheaf) if and only if for every manifold $M$ and every open cover $\mathcal{U}=\left(U_\alpha \hookrightarrow M\right),$ the canonical map
$$\Hom\left(y\left(M\right),\X\right) \to \Hom\left(\tilde y\left(M_\mathcal{U}\right),\X\right)$$ is an equivalence of groupoids. Denote the full subcategory of $\Psh_1\left(\Mfd\right)$ on the stacks by $\St\left(\Mfd\right)$. The canonical inclusion $$\St\left(\Mfd\right) \hookrightarrow \Psh_1\left(\Mfd\right)$$ admits a left adjoint $a$ called the \textbf{stackification} functor, and $a$ preserve finite limits.
\end{dfn}

\begin{rmk}
One can analogously define stacks on the large category $\widetilde{\Mfd},$ and the resulting bicategory is canonically equivalent to the above, by Remark \ref{rmk:compy}.
\end{rmk}

\begin{dfn}
We denote by $\left[\G\right]$ the associated stack on $\Mfd$, $a \circ \tilde y\left(\G\right)$, where $a$ is the stackification $2$-functor. The stack $\left[\G\right]$ is called the \textbf{stack completion} of the groupoid $\G$.
\end{dfn}

\begin{dfn}
A stack of groupoids $\X$ on $\Mfd$ is a \textbf{differentiable stack} if it is equivalent to $\left[\G\right]$ for some Lie groupoid $\G$.
\end{dfn}

\begin{dfn}
If $G \acts M$ is a smooth action of a Lie group on a manifold, then we denote by $M//G$ the stack completion of the action groupoid $G \ltimes M,$ and call it the \textbf{stacky-quotient} of $M$ by $G$.
\end{dfn}

\begin{prop}\label{prop: Morita1}
If $\varphi:\H \to \G$ is a smooth functor, then the induced map $$\left[\varphi\right]:\left[\H\right] \to \left[\G\right]$$ is an equivalence if and only if $\varphi$ is a Morita equivalence.
\end{prop}

\begin{dfn}\label{dfn:etalestack}
An \textbf{\'etale differentiable stack} is a stack $\X$ of groupoids on $\Mfd$ such that $\X \simeq \left[\G\right]$ for $\G$ an \'etale Lie groupoid. We denote the full subcategory of $\St\left(\Mfd\right)$ on the \'etale differentiable stacks by $\Etdsd.$
\end{dfn}

\begin{rmk}
By Proposition \ref{prop:folgpd} and Proposition \ref{prop: Morita1}, it follows that a stack $\X$ of groupoids on $\Mfd$ is an \'etale differentiable stack if and only if $\X$ is equivalent to $\left[\G\right]$ for $\G$ a foliation Lie groupoid.
\end{rmk}

\begin{dfn}
An \textbf{orbifold} is a stack $\X$ of groupoids on $\Mfd$ such that $\X \simeq \left[\G\right]$ for $\G$ a proper foliation Lie groupoid.
\end{dfn}

\begin{dfn}\label{dfn:Hafstack}
Let $\Gamma^n$ denote the \'etale Lie groupoid from Example \ref{ex:Haefliger}. Then $\left[\Gamma^n\right]=:\HA_n$ is called the \textbf{$n$-dimensional Haefliger stack.}
\end{dfn}

Every smooth functor $\H \to \G$ induces a map $\left[\H\right] \to \left[\G\right]$ and the induced functor $$\Hom\left(\H,\G\right) \to \Hom\left(\left[\H\right],\left[\G\right]\right)$$ is full and faithful, but not in general essentially surjective. However, any morphism $$\left[\H\right] \to \left[\G\right]$$ arises from a chain $$\H \leftarrow \K \rightarrow \G,$$ with $\K \to \H$ a Morita equivalence. In fact, the class of Morita equivalences admits a calculus of fractions, and the bicategory of differentiable stacks is equivalent to the bicategory of fractions of Lie groupoids with inverted Morita equivalences. For details see \cite{Dorette}.

\subsection{Higher \'Etale Differentiable Stacks}

\begin{dfn}\label{dfn:ipsheave}
An $\i$-presheaf on manifolds is a functor $\X:\Mfd^{op} \to \iGpd$ from (the opposite of) the category of smooth manifolds to the $\icat$ of $\i$-groupoids. We denote the $\i$-category of such functors by $\Pshi\left(\Mfd\right).$ Given any weak presheaf of groupoids $\X,$ there is an associated $\i$-presheaf $$\Mfd^{op} \stackrel{\X}{\longrightarrow} \Gpd \hookrightarrow \iGpd,$$ and this assignment induces a canonical inclusion $$\Psh_1\left(\Mfd\right) \hookrightarrow \Pshi\left(\Mfd\right).$$
\end{dfn}

\begin{rmk}
Any $\i$-presheaf can be modeled by an ordinary functor $$X:\Mfd^{op} \to \Set^{\Delta^{op}}$$ from (the opposite of) the category of manifolds to the category of simplicial sets, such that each simplicial set $X\left(M\right)$ is a Kan complex. More precisely, there is a model category structure on simplicial presheaves whose fibrant objects satisfy the above condition, and whose associated $\i$-category is equivalent to $\Pshi\left(\Mfd\right).$
\end{rmk}

\begin{dfn}
We say that an $\i$-presheaf on manifolds $\X$ is an $\i$-stack (or $\i$-sheaf) if for any manifold $M$ and any open cover $\left(U_\alpha \hookrightarrow M\right),$ the canonical map $$\Hom\left(y\left(M\right),\X\right) \to \Hom\left(\tilde y\left(M_\mathcal{U}\right),\X\right)$$ is an equivalence of $\i$-groupoids. Denote the full subcategory of $\Pshi\left(\Mfd\right)$ on the $\i$-stacks by $\Shi\left(\Mfd\right).$ The canonical inclusion $$\Shi\left(\Mfd\right) \hookrightarrow \Pshi\left(\Mfd\right)$$ admits a left adjoint $a$ called the \textbf{$\i$-stackification} functor, and $a$ preserve finite limits.
\end{dfn}

By Proposition \ref{prop:ithesame}, one can define $\i$-stacks over $\widetilde{\Mfd}$ in a completely analogous way, and there is a canonical equivalence of $\i$-categories $\Shi\left(\widetilde{\Mfd}\right) \simeq \Shi\left(\Mfd\right).$

\begin{dfn}
Denote by $\widetilde{\Mfd}^{\et}$ the category of (possibly non-Hausdorff or $2^{nd}$ countable) smooth manifolds and their local diffeomorphisms. An \textbf{\'etale simplicial manifold} is a functor $$\G_\bullet:\Delta^{op} \to \widetilde{\Mfd}^{\et}.$$ Such an \'etale simplicial manifold is said to be $n$-dimensional if each $\G_i$ is an $n$-dimensional manifold.
\end{dfn}

An \'etale simplicial manifold $\G_\bullet$ yields a simplicial $\i$-sheaf
$$\tilde \G_\bullet:\Delta^{op} \to \Shi\left(\Mfd\right),$$
which is defined as the composite
$$\Delta^{op} \stackrel{\G_\bullet}{\longlonglongrightarrow} \widetilde{\Mfd}^{\et} \to \widetilde{\Mfd}  \stackrel{y}{\hookrightarrow} \Shi\left(\widetilde{\Mfd}\right) \simeq \Shi\left(\Mfd\right).$$
Denote by $$\left[\G_\bullet\right]:=\colim \tilde \G_\bullet.$$

\begin{rmk}\label{rmk:Gcolim}
If $\G_\bullet$ is the nerve of an \'etale Lie groupoid $\G$, then $\left[\G_\bullet\right]=\left[\G\right]$ is the \'etale differentiable stack associated to $\G.$ A general \'etale simplicial manifold should be thought of as a higher \'etale Lie groupoid. To see that  $\left[\G_\bullet\right]=\left[\G\right]$, it suffices to see that the colimit of $\G_\bullet$ in $\i$-presheaves is $\tilde y\left(\G\right),$ since $a$ preserves colimits, as it is a left adjoint. Colimits in $\i$-presheaves are computed object-wise, so it suffices to see that the nerve of a discrete groupoid is the homotopy colimit of itself, considered as a simplicial diagram of simplicial sets. This follows from the well known fact that the homotopy colimit of a bisimplicial set can be computed as its diagonal. 
\end{rmk}

\begin{dfn}
An $\i$-stack $\X$ on $\Mfd$ is an \textbf{\'etale differentiable $\i$-stack} if and only if it is equivalent to one of the form $\left[\G_\bullet\right]$ with $\G_\bullet$ an \'etale simplicial manifold. Denote the corresponding $\icat$ by $\Etds.$ Such an \'etale differentiable $\i$-stack is said to be $n$-dimensional if $\X \simeq \left[\G_\bullet\right],$ with $\G_\bullet$ $n$-dimensional. Denote the corresponding $\icat$ by $\Etdsn.$
\end{dfn}

\begin{rmk}
In \cite{higherme}, we define \'etale differentiable $\i$-stacks as $\i$-stacks arising as the functor of points of certain locally ringed $\i$-topoi. However, Proposition 6.1.2 of loc. cit. establishes an equivalence between these two definitions.
\end{rmk}

\begin{dfn}\label{dfn:prol}
Denote by $\Mfd^{\et}$ the category of manifolds and their local diffeomorphisms and by $$j:\Mfd^{\et} \to \Mfd$$ the canonical functor. There is a canonical restriction functor
$$j^*:\Shi\left(\Mfd\right) \to \Shi\left(\Mfd^{\et}\right)$$ and it admits a left adjoint $j_!$ called the \textbf{\'etale prolongation} functor. Similarly, denote by $$j_n:n\mbox{-}\Mfd^{\et} \to \Mfd$$ the canonical functor from $n$-manifolds and their local diffeomorphisms to all smooth manifolds and all smooth maps. The restriction functor $$j_n^*:\Shi\left(\Mfd\right) \to \Shi\left(\Mfd^{\et}\right)$$ likewise admits a left adjoint $j^n_!$ called the \textbf{$n$-dimensional \'etale prolongation functor.}
\end{dfn}

\begin{thm}\label{thm:5.3.9} (\cite{higherme} Theorem 5.3.9)
An $\i$-stack $\X$ on $\Mfd$ is an \'etale differentiable $\i$-stack if and only if it is in the essential image of the \'etale prolongation functor $j_!$. Similarly, an $\i$-stack $\X$ on $\Mfd$ is an $n$-dimensional \'etale differentiable $\i$-stack if and only if it is in the essential image of the $n$-dimensional \'etale prolongation functor $j^n_!$.
\end{thm}

The concept of a local diffeomorphism of manifolds extends to \'etale differentiable $\i$-stacks. This is done most naturally by considering such stacks as the functor of points of smooth $\i$-\'etendues and then appealing to the notion of a local homeomorphism of structured $\i$-topoi, however this can be defined in another way:

\begin{dfn}
A morphism $f:\X \to \Y$ between \'etale differentiable $\i$-stacks is a local diffeomorphism if it is in the essential image of $j_!$.
\end{dfn}

\begin{rmk}\label{rmk:localdiffeo}
For $\X$ and $\Y$ \'etale differentiable stacks with $\Y \simeq \left[\H\right],$ $$f:\X \to \Y$$ is a local diffeomorphism if and only if there exists an \'etale Lie groupoid $\G$ with $\left[\G\right] \simeq \X$ and a smooth functor $F:\G \to \H$ such that $F_0:\G_0 \to \H_0$ is a local diffeomorphism, such that $f$ is equivalent to $\left[F\right].$
\end{rmk}

The $n$-dimensional Haefliger stack of Definition \ref{dfn:Hafstack} enjoys a special universal property with respect to local diffeomorphisms:

\begin{thm} (\cite{higherme} Theorem 6.1.6) \label{thm:6.1.6}
Denote by $\Etdsn^{\et}$ the $\i$-category consisting of $n$-dimensional \'etale differentiable $\i$-stacks and only the local diffeomorphisms between them. Then $\HA_n$ is the terminal object in this $\i$-category, that is, for any other $n$-dimensional \'etale differentiable $\i$-stacks $\X,$ the $\i$-groupoid $\Hom^{\et}\left(\X,\HA_n\right)$ of local diffeomorphisms from $\X$ to $\HA_n$ is contractible.
\end{thm}

Given an $n$-dimensional \'etale differentiable $\i$-stacks $\X,$ it determines an $\i$-stack $y^{\et}\left(\X\right)$ on $\nMfd^{\et}$ by assigning to each $n$-manifold $M$ the $\i$-groupoid $\Hom^{\et}\left(M,\X\right)$ of local diffeomorphisms from $M$ to $\X.$ The assignment $\X \mapsto y^{\et}\left(\X\right)$ is not functorial with respect to all maps of stacks, but it is with respect to local diffeomorphisms and ones gets an induced functor
$$y^{\et}:\Etdsn^{\et} \to \Shi\left(\nMfd^{\et}\right).$$

\begin{thm} (\cite{higherme} Theorem 6.1.3) \label{thm:6.1.3}
The above functor $y^{\et}$ is an equivalence of $\i$-categories. Moreover, given an $n$-dimensional \'etale stack $\X,$ there is a canonical equivalence $\X \simeq j^n_! y^{\et}\left(\X\right).$
\end{thm}

\begin{cor} \label{cor:ex1}
The $n$-dimensional Haefliger stack $\HA_n$ can be described as $$\HA_n \simeq j^n_! \left(1\right),$$ where $1$ is the terminal sheaf.
\end{cor}

\begin{proof}
By Theorem \ref{thm:6.1.6}, $y^{\et}\left(\HA_n\right)$ must be terminal. The result now follows from Theorem \ref{thm:6.1.3}.
\end{proof}

\begin{ex}\label{ex:symprol}
Consider the functor $$\mathcal{S}_{2n}:\left(2n\mbox{-}\Mfd^{\et}\right)^{op} \to \Set$$ sending a $2n$-dimensional manifold $M$ to the set of symplectic forms on $M$. Note that this is not even a functor on $2n\mbox{-}\Mfd$, but it is functorial with respect to local diffeomorphisms, and in fact is a sheaf. By Example 2 of Section 3.2.2 of \cite{prol}, $$j^{2n}_!\left(\mathcal{S}_{2n}\right)\simeq \left[\Gamma^{Sp}_{2n}\right]$$ where $\Gamma^{Sp}_{2n}$ is the Lie groupoid from Example \ref{ex:sympgpd}.
\end{ex}

\begin{ex}\label{ex:riemprol}
Let $\mathcal{R}_n:\left(\nMfd^{\et}\right)^{op} \to \Set$ be the functor which assigns each $n$-manifold $M$ its set of Riemannian metrics. Then similarly to the above example, although this is not functorial with respect to all maps, it is functorial with respect to local diffeomorphisms and is a sheaf. By Example 3 of Section 3.2.2 of \cite{prol}, we have that $$j^n_!\left(\mathcal{R}_n\right)\simeq \left[\cR\Gamma^n\right].$$
\end{ex}

\section{The Homotopy Type of a Higher Stack}\label{sec:homotopy type}

\subsection{The Fundamental Infinity Groupoid Functor}\label{sec:funigpd}
We start this section with a very important proposition. It is not due to us. It follows as a formal corollary from Proposition 8.3 of \cite{dugger}, and also from Proposition 4.3.29 and Corollary 4.4.28 of \cite{dcct}. However, we will give a simpler and more direct argument. We will also derived a slightly stronger statement by similar techniques.

\begin{prop}\label{prop:pi}
There is a colimit preserving functor $$\Pi_\i:\Shi\left(\Mfd\right) \to \iGpd$$ which sends every manifold $M$ in $\Mfd$ to its underlying homotopy type. For a given $\i$-stack $F$, we call $\Pi_\i\left(F\right)$ its \textbf{fundamental $\i$-groupoid}.
\end{prop}

First we will need an easy lemma:

\begin{lem}\label{lem:h}
There is a canonical functor $\Top \to \iGpd$ from the category of topological spaces to the $\i$-category of $\i$-groupoids sending each space $X$ to its weak homotopy type.
\end{lem}

\begin{proof}
Let $\Top$ denote the category of topological spaces, and let $W$ denote the class of weak homotopy equivalences. There exists an $\icat$ $\C$ together with a morphism $h:\Top \to \C,$ universal with the property that for any $\i$-category $\sD,$ composition with $h$ induces an equivalences of $\i$-categories $$\Fun\left(\C,\sD\right) \to \Fun_W\left(\Top,\sD\right)$$ between the $\icat$ of functors $\C \to \sD$ and the $\i$-category of functors $\Top \to \sD$ which send each weak homotopy equivalence to an equivalence in $\sD.$ The existence follows from Section 2.1 of \cite{hinich}. By Proposition 2.2.1 of op. cit., it follows that $\C$ must be the $\i$-category associated to the standard Quillen model structure on $\Top,$ which is none other than $\iGpd.$ We conclude that there is a canonical functor $h:\Top \to \iGpd,$ sending each topological space to its weak homotopy type.
\end{proof}

\begin{rmk}\label{rmk:colimok}
By the proof of Corollary 4.2.4.8 of \cite{htt}, it follows that $h$ sends homotopy colimits in $\Top$ to colimits in $\iGpd.$
\end{rmk}

We will now give a proof of Proposition \ref{prop:pi}:

\begin{proof}
From Lemma \ref{lem:h}, by composition we get a functor $$\pi:\Mfd \to \Top \to \iGpd$$ sending each smooth manifold to its homotopy type. By Theorem \ref{thm:5.1.5.6}, by left Kan extension there is a colimit preserving functor $$\Lan_y \pi:\Pshi\left(\Mfd\right) \to \iGpd$$ sending every manifold to its underlying homotopy type. It follows from the Yoneda lemma that this functor has a right adjoint $R_\pi$ which sends an $\i$-groupoid $X$ to the $\i$-presheaf $$R_\pi\left(X\right):M \mapsto \Hom\left(\pi\left(M\right),X\right).$$ We claim that $R_\pi\left(X\right)$ is an $\i$-sheaf. To see this, it suffices to observe that if $$\mathcal{U}=\left(U_\alpha \hookrightarrow M\right)$$ is an open cover of a manifold $M,$ then the colimit of $$\Delta^{op} \stackrel{N\left(M_\mathcal{U}\right)}{\longlongrightarrow} \Mfd \to \iGpd$$ is $\pi\left(M\right),$ which follows for an open cover of an arbitrary topological space by Theorem 1.1 of \cite{duggerisaksen} combined with Remark \ref{rmk:colimok}. It follows that $\Lan_y \pi$ and $R_\pi$ restrict to an adjunction $$\Pi_\infty \dashv \Delta$$ between $\Shi\left(\Mfd\right)$ and $\iGpd,$ so in particular, $\Pi_\infty$ preserves colimits.
\end{proof}

An advantage of our approach in regards to that of \cite{dugger} and \cite{dcct} is it readily generalizes, e.g.:

\begin{prop}\label{prop:pi2}
The functor $$\Pi_\i:\Shi\left(\Mfd\right) \to \iGpd$$ of Proposition \ref{prop:pi} sends every manifold $M$ in $\widetilde{\Mfd}$ to its underlying homotopy type. (Recall that $\widetilde{\Mfd}$ is the category of all topological spaces with a smooth atlas, i.e. smooth manifolds which need not be Hausdorff or $2^{nd}$-countable.)
\end{prop}

\begin{proof}
Denote by $\LiGpd$ the $\icat$ of $\i$-groupoids belonging to a larger Grothendieck universe $\mathcal{V}.$ Then by the same argument as for the proof of \ref{prop:pi} there exists a $\mathcal{V}$-small colimit preserving functor $\widehat{\Pi_\i}:\LShi\left(\widetilde{\Mfd}\right) \to \LiGpd$ which sends every manifold $M$ in $\widetilde{\Mfd}$ to its underlying homotopy type. Note that by Remark 6.3.5.17 of \cite{htt}, both inclusions $$\iGpd \hookrightarrow \LiGpd$$ and $$\Shi\left(\widetilde{\Mfd}\right) \hookrightarrow \LShi\left(\widetilde{\Mfd}\right)$$ preserve small colimits (with respect to the smaller universe). It follows that the two composites
$$\Shi\left(\Mfd\right) \simeq \Shi\left(\widetilde{\Mfd}\right) \hookrightarrow \LShi\left(\widetilde{\Mfd}\right) \stackrel{\widehat{\Pi_\i}}{\longlonglongrightarrow} \LiGpd$$ and
$$\Shi\left(\Mfd\right) \stackrel{\Pi_\infty}{\longlongrightarrow} \iGpd \hookrightarrow \LiGpd$$ are colimit preserving and agree (up to homotopy) on every manifold $N$ in $\Mfd.$ It follows from Theorem \ref{thm:5.1.5.6} that both functors must in fact be equivalent. In particular, if $M$ is a manifold in $\widetilde{\Mfd},$ then $\Pi_\i\left(M\right) \simeq \widehat{\Pi_\i}\left(M\right).$
\end{proof}

\subsection{Fat Geometric Realization}\label{sec:fat}

We will now explain the relationship between the functor $$\Pi_\i:\Shi\left(\Mfd\right) \to \iGpd$$ and the fat geometric realization functor $$|| \mspace{3mu} \bullet \mspace{3mu}||:\Top^{\Delta^{op}} \to \Top.$$

\begin{dfn}\label{dfn:fat}
Let $\Delta_+$ denote the category of non-empty finite ordinals and injective order-preserving maps. A \textbf{semi-simplicial} space is a functor $$Y_\bullet:\Delta^{op}_+ \to \Top.$$ Denote by $$i:\Delta_+ \to \Delta$$ the canonical functor, and denote by $$i^*:\Top^{\Delta^{op}} \to \Top^{\Delta_+^{op}}$$ the functor $X_\bullet \mapsto X_\bullet \circ i^{op}.$ The \textbf{fat geometric realization} of a semi-simplicial space $Y_\bullet$ is the  following co-end:
$$||Y|| =\int^{n \in \Delta_+} Y_n \times \Delta^n,$$ and the fat geometric realization of a simplicial space $X_\bullet$ is by definition the fat geometric realization of $i^*X_\bullet,$ which we will also denote by $||X||.$ More concretely, the fat geometric realization of a simplicial space $X_\bullet$ is the same as the ordinary geometric realization $|X|,$ except one does not quotient out $$\coprod_n{X_n \times \Delta^n}$$ by the relations induced by the degeneracy maps.
\end{dfn}

We begin by recalling a point-set theoretic definition from Appendix A of \cite{duggerisaksen}:
\begin{dfn}\label{dfn:relt1}
An embedding $Y \hookrightarrow Z$ of topological spaces is said to be \textbf{relatively $T1$} if for any open subset $U \subseteq Y,$ and any point $z \in Z\backslash U,$ there is an open subset $W \subseteq Z$ such that $z \notin W$ and $U \subseteq W.$
\end{dfn}

\begin{prop} (Lemma A.2 of \cite{duggerisaksen})\label{prop:A2}
If
$$\xymatrix{X \times \partial\Delta^n \ar[r] \ar[d] & Y \ar[d]\\
X \times \Delta^n \ar[r] & Z}$$ is a pushout diagram in $\Top$, then the embedding $Y \hookrightarrow Z$ is relatively $T1$.
\end{prop}

We will need two lemmas. It seems that they are folklore in some circles, however we found that their complete proofs, without making any topological assumptions such as being $T1$, to be worthy of a careful treatment. The outline of the proof of both we learned from Danny Stevenson, and we also use many ideas from Appendix A of \cite{duggerisaksen}.



\begin{lem}\label{lem:degwise}
If $f:X_\bullet \to Y_\bullet$ is a map of (semi-)simplicial spaces which is a degree-wise weak homotopy equivalence, then the induced map $$||X|| \to ||Y||$$ is a weak homotopy equivalence.
\end{lem}

\begin{proof}
Given a simplicial space $X_\bullet$ there is a filtration 
$$||X||=\underset{n}\colim\sk_n||X||,$$ where the $n$-skeleta $\sk_n||X||$ are defined inductively as 
the pushouts
$$\xymatrix{
X_n\times\partial\Delta^n\ar[r]\ar[d] & \sk_{n-1}||X||\ar[d]\\
X_n\times\Delta^n\ar[r]       & \sk_n ||X||.}$$
The map $X_n\times\partial\Delta^n \hookrightarrow X_n\times \Delta^n$ may fail to be a Serre cofibration, but is a closed Hurewicz cofibration. Closed Hurewicz cofibration are the cofibrations with respect to the Str\o{}m model structure on $\Top$ (c.f. \cite{strom}). In this model structure, every topological space is cofibrant, from which it follows that the Str\o{}m model structure is left proper, and hence each map $$\sk_{n-1}||X|| \hookrightarrow \sk_n||X||$$ is again a Str\o{}m cofibration. It follows that $$||X||=\colim \sk_n ||X||$$ is a homotopy colimit with respect to the Str\o{}m model structure, and hence also with respect to the standard Quillen model structure by Lemma A.7 of \cite{duggerisaksen}. Similarly the above pushout diagram defining $\sk_n ||X||$ is also a homotopy pushout diagram with respect to the Quillen model structure.

We now make two observations:
\begin{itemize}
\item[1)] Any map from a sphere $S^n$ or disk $D^{n+1}$ into $||X||$ must factor through at a finite stage of the filtration by skeleta.
\item[2)] Each induced map $$\sk_{n-1}||X|| \hookrightarrow \sk_n||X||$$ induces an isomorphism on $\pi_i$ for $i<n$.
\end{itemize}
The first observation is mostly standard, however the usual argument usually assumes that each stratum of the filtration is a $T1$-space, however, one actually only needs that the inclusion of each stratum is relatively $T1$ in the sense of Definition \ref{dfn:relt1}, which follows from Lemma A.2 of \cite{duggerisaksen}; see Lemma A.3 of \cite{duggerisaksen} for details. The second observation follows immediately from the fact that for all $n$, we have a canonical identification $$\sk_n ||X||/\sk_{n-1} ||X|| \cong \Sigma^n X_n,$$ which is an $\left(n-1\right)$-connected space.

We now claim that for all $n$, the induced map $$j_n:\sk_{n}||X|| \hookrightarrow ||X||$$ induces an isomorphism on $\pi_i$ for $i \le n$. To see that the map is surjective, let $\left[\gamma\right] \in \pi_i\left(||X||\right),$ and consider a map $$f:S^i \to ||X||$$ such that $\left[f\right]=\left[\gamma\right].$ Then $f$ must factor through an inclusion of the form $$\sk_{k}||X|| \hookrightarrow ||X||$$ for some $k$. If $k > n$ then it follows by induction from $2)$ above, that $$\sk_{n}||X|| \hookrightarrow \sk_k||X||$$ induces an isomorphism on $\pi_i,$ and hence $f$ is homotopic to a map factoring through the inclusion $j_n$. To see that the map is injective, suppose that $$\left[\alpha\right] \in \pi_i\left(\sk_{n}||X||\right)$$ and $$\left(j_n\right)_*\left(\left[\alpha\right]\right)=0.$$ This means that the map $S^i \stackrel{\alpha}{\longrightarrow} \sk_{n}||X|| \stackrel{j_n}{\longlongrightarrow} ||X||$ extends to a continuous map $\tilde \alpha:D^{i+1} \to ||X||.$ The map $\tilde \alpha$ must factor through an inclusion of the form $$\sk_{k}||X|| \hookrightarrow ||X||$$ for some $k.$ If $k \le n,$ then it follows that $\alpha$ is null-homotopic, and if $k >n$ it follows that $\left(j_{n,k}\right)_*\left[\alpha\right]=0,$ where $$j_{n,k}:\sk_{n}||X|| \hookrightarrow \sk_k||X||$$ is the  canonical inclusion. By $2)$, $\left(j_{n,k}\right)_*$ is an isomorphism, so we must have $\left[\alpha\right]=0.$ 

From the above, it follows that to show that the map $$||f||:||X|| \to ||Y||$$ is a weak homotopy equivalence, it suffices to show that for all $n,$ the induced map $$\sk_n ||X|| \to \sk_n ||Y||$$ is a weak homotopy equivalence. We will prove this by induction on $n$. Since $\sk_0 ||X||=X_0,$ this establishes the base case of our induction. Now suppose that  $$\sk_{n-1} ||X|| \to \sk_{n-1}||Y||$$ is a weak homotopy equivalence. We have a map of spans which consists of weak homotopy equivalences:
$$\xymatrix{X_n\times\Delta^n \ar[d] & X_n\times\partial\Delta^n \ar[l] \ar[d] \ar[r] & \sk_{n-1}||X|| \ar[d]\\
Y_n\times\Delta^n& Y_n\times\partial\Delta^n \ar[l] \ar[r] & \sk_{n-1}||Y||,}$$
from which it follows that the induced map between their homotopy pushouts is a weak equivalence, and since their homotopy pushouts are given by $\sk_n ||X||$ and $\sk_n ||Y||$ respectively, this establishes our claim.

\end{proof}

\begin{lem}\label{lem:fathoc}
If $X_\bullet:\Delta^{op} \to \Top$ is any simplicial space then $||X||$ is the homotopy colimit of $X_\bullet.$
\end{lem}

\begin{proof}
Consider the simplicial space $Y_\bullet$ defined by $$Y_n=|\operatorname{Sing}\left(X_n\right)|,$$ together with the canonical map $$\epsilon: Y_\bullet \to X_\bullet$$ which is degree-wise a weak homotopy equivalence by construction. It follows that the induced map $$\hocolim Y_\bullet \to \hocolim X_\bullet$$ is a weak homotopy equivalence. Notice that the geometric realization functor $$|\mspace{3mu}\bullet \mspace{3mu}|:\Set^{\Delta^{op}} \to \Top$$ is left Quillen with respect to the standard Quillen model structure on both sides, hence preserves homotopy colimits. It follows that $$\hocolim Y_\bullet \simeq | \hocolim \operatorname{Sing}\left(X_\bullet\right)|.$$ Since the homotopy colimit of a simplicial diagram of simplicial sets can be computed as the diagonal of the resulting bisimplicial set, we have 

$$\hocolim \operatorname{Sing}\left(X_\bullet\right) \simeq  \operatorname{diag}\left(\operatorname{Sing}\left(X_\bullet\right)_\bullet\right)$$ from which it follows that

$$\hocolim Y_\bullet \simeq |\operatorname{diag}\left(\operatorname{Sing}\left(X_\bullet\right)_\bullet\right)|\cong | \left(|\operatorname{Sing}\left(X_\bullet\right)_\bullet\right)|,$$
where the isomorphism follows from the Eilenberg-Zilber theorem.

Note that $CW$-complexes are locally equi-connected and by Lemma 3.1 $a)$ of \cite{lewis}, inclusions of retracts of locally equi-connected spaces are closed cofibrations, so it follows that the simplicial space $Y_\bullet = |\operatorname{Sing}\left(X_\bullet\right)|$ is \emph{good} in the sense of Appendix A of \cite{catcoh}. It follows now from Proposition A.1 (iv) of the same appendix that the induced map
$$||Y|| \to |Y|=| \left(|\operatorname{Sing}\left(X_\bullet\right)_\bullet\right)|$$ is a weak homotopy equivalence. Hence, $||Y||$ is also the homotopy colimit of $Y_\bullet.$ Finally, since the map $$\epsilon: Y_\bullet \to X_\bullet$$ is degree-wise a weak homotopy equivalence, by Lemma \ref{lem:degwise}, the induced map $$||Y|| \to ||X||$$ is a weak homotopy equivalence, and thus $||X||$ is the homotopy colimit of $Y_\bullet,$ and hence of $X_\bullet,$ as desired.
\end{proof}




\begin{thm}
Suppose that $X_\bullet:\Delta^{op} \to \widetilde{\Mfd}$ is a simplicial manifold and let $\left[X_\bullet\right]:=\colim \tilde X_\bullet,$ where $\tilde X_\bullet$ is defined as the composite $$\Delta^{op} \stackrel{X_\bullet}{\longlongrightarrow} \widetilde{\Mfd}  \stackrel{y}{\hookrightarrow} \Shi\left(\widetilde{\Mfd}\right) \simeq \Shi\left(\Mfd\right).$$ Then $$\Pi_\i \left[X_\bullet\right] \simeq h \left(||X||\right),$$ where $h:\Top \to \iGpd$ is the functor from Lemma \ref{lem:h}.
\end{thm}

\begin{proof}
By Remark \ref{rmk:colimok}, the functor $h$ sends homotopy colimits to colimits, hence from Lemma \ref{lem:fathoc} and Proposition \ref{prop:pi2}, it follows that 

\begin{eqnarray*}
h \left(||X||\right) &\simeq& \colim \left(\ldots h\left(X_2\right) \rrrarrow h\left(X_1\right) \rrarrow h\left(X_0\right)\right)\\
&\simeq& \colim \left(\ldots \Pi_\i\left(X_2\right) \rrrarrow \Pi_\i\left(X_1\right) \rrarrow \Pi_\i\left(X_0\right)\right)\\
&\simeq& \Pi_\i \left[X_\bullet\right],
\end{eqnarray*}
the final equivalence following from the fact that $\Pi_\i$ preserves colimits.
\end{proof}

\begin{cor}\label{cor:diffhom}
For $\left[\G\right]$ a differentiable stack, $$\Pi_\i\left(\left[\G\right]\right) \simeq h\left(B\G\right).$$
\end{cor}

\subsection{Sheaves over The Monoid of Embeddings of $\R^n$}\label{sec:monoid1}

Denote by $\Emb$ the discrete monoid of smooth embeddings of $\RR^n$ into itself. Regarding this monoid as a category with one object, there is a canonical functor
$$\alpha:\Emb \to \nMfd^{\et}$$ to the category of $n$-manifolds and local diffeomorphisms, sending the unique object to $\RR^n$ and each embedding to itself, considered as a local diffeomorphism $$\RR^n \to \RR^n.$$

We can introduce a Grothendieck pre-topology on the category $\Emb$ where a collection of embeddings $$\left(\varphi_i:\RR^n \hookrightarrow \RR^n\right)$$ is a cover when the family is jointly surjective.

\begin{rmk}\label{rmk:notsubcan}
The Grothendieck topology associated to the above pre-topology is not subcanonical. The essential image of
$$\Emb \stackrel{y}{\hookrightarrow} \Psh\left(\Emb\right) \stackrel{a}{\longrightarrow} \Sh\left(\Emb\right)$$ is canonically equivalent to the monoid $\operatorname{LocDiff}\left(\RR^n\right)$ of self local diffeomorphisms of $\RR^n.$
\end{rmk}

The following theorem is proven in Appendix \ref{sec:monoid}:

\begin{customthm}{\ref{thm:embeql}}
The canonical functor $$\alpha^*:\Pshi\left(\nMfd^{\et}\right) \to \Pshi\left(\Emb\right)$$ restricts to an equivalence of $\i$-categories
$$\alpha^*:\Shi\left(\nMfd^{\et}\right) \to \Shi\left(\Emb\right).$$
\end{customthm}

\subsection{The Homotopy Type of Higher \'Etale Differentiable Stacks}\label{sec:main}

By Theorem \ref{thm:5.3.9}, for every $n$-dimensional \'etale differentiable $\i$-stack $\X,$ there exists a (unique) $\i$-stack $F$ on the site of $n$-manifolds and their local diffeomorphisms such that $j^n_!F \simeq \X,$ where $j^n_!$ is the $n$-dimensional \'etale prolongation functor of Definition \ref{dfn:prol}. Since $F$ determines $\X,$ it is natural to ask if the homotopy type $\Pi_\i\left(\X\right)$ can be naturally expressed in terms of $F.$ The following result answers this question:

\begin{thm}\label{thm:main}
Suppose that $F$ in $\Shi\left(\nMfd^{\et}\right)$ is an $\i$-sheaf on $n$-manifolds and their local diffeomorphisms, and let $\X=j^n_!F$ be its associated $n$-dimensional \'etale differentiable $\i$-stack. Then the $\i$-groupoid $\Pi_\i\left(\X\right)$ can be expressed as the colimit of the following composite:

$$\Emb^{op} \stackrel{\alpha^{op}} {\longlongrightarrow} \left(\nMfd^{\et}\right)^{op} \stackrel{F}{\longrightarrow} \iGpd.$$
\end{thm}

\begin{proof}
Since $j_!^n \dashv j_n^*,$ $j_!^n$ is colimit preserving. Denote by $a$ the $\i$-sheafification functor $a:\Pshi\left(\Emb\right) \to \Shi\left(\Emb\right),$ i.e. $a$ is the left adjoint to the inclusion $$i:\Shi\left(\Emb\right) \hookrightarrow \Pshi\left(\Emb\right)$$ of $\i$-sheaves on $\Emb$ into $\i$-presheaves on $\Emb.$ Finally, denote by $\alpha_!$ a left adjoint to the equivalence $\alpha^*$ of Theorem \ref{thm:embeql} (which must exist by Corollary 5.5.2.9 of \cite{htt}, and must also be an equivalence).
It follows that the following composite is a colimit preserving functor:
$$\resizebox{5in}{!}{$\Pshi\left(\Emb\right) \stackrel{a}{\longrightarrow} \Shi\left(\Emb\right) \stackrel{\alpha_!}{\longlongrightarrow} \Shi\left(\nMfd^{\et}\right) \stackrel{j^n_!}{\longlongrightarrow} \Shi\left(\Mfd\right) \stackrel{\Pi_\i}{\longlongrightarrow} \iGpd.$}$$
Since $\alpha_! \dashv \alpha^*$ is an adjoint equivalence, it follows from Remark \ref{rmk:notsubcan} that if $$y_{\Emb}:\Emb \hookrightarrow \Pshi\left(\Emb\right)$$ and $$y^{\et}:\nMfd^{\et} \hookrightarrow \Shi\left(\nMfd^{\et}\right)$$ denote the Yoneda embeddings, that $$\alpha_!a \left(y_{\Emb}\left(\R^n\right)\right) \simeq y^{\et}\left(\R^n\right),$$ and hence
\begin{eqnarray*}
\Pi_\i\left(j_!^n\left(\alpha_!\left(a\left(y\left(\R^n\right)\right)\right)\right)\right) &\simeq& \Pi_\i\left(j_!^n\left(y^{\et}\left(\R^n\right)\right)\right)\\
&\simeq& \Pi_\i\left(y\left(\R^n\right)\right)\\
&\simeq& h\left(\R^n\right) \simeq *.
\end{eqnarray*}
In other words, the composite $$\Emb \stackrel{\Pi_\i \circ j_!^n \circ \alpha_! \circ a \circ y_{\Emb}}{\longlonglonglonglongrightarrow} \iGpd$$ is the terminal functor, i.e. equivalent to the constant functor $$t:\Emb \to \iGpd$$ with value the contractible $\i$-groupoid $*.$ By Theorem \ref{thm:5.1.5.6}, this implies that
\begin{eqnarray*}
\Pi_\i \circ j_!^n \circ \alpha_! \circ a&\simeq& \Lan_{y_{\Emb}}\left(t\right)\\
&\simeq& \colim\left(\mspace{3mu} \cdot \mspace{3mu}\right),
\end{eqnarray*}
where $$\colim\left(\mspace{3mu} \cdot \mspace{3mu}\right):\Pshi\left(\Emb\right) \to \iGpd$$ is the functor sending an $\i$-presheaf $$F:\Emb^{op} \to \iGpd$$ to its colimit (see Proposition \ref{prop:colimlan}). So finally, we have
\begin{eqnarray*}
\Pi_\i\left(\X\right) &\simeq& \Pi_\i\left(j^n_!F\right)\\
&\simeq& \Pi_\i\left(j^n_!\alpha_!\alpha^*F\right)\\
&\simeq& \Pi_\i\left(j^n_! \alpha_! a\left(F\circ \alpha^{op}\right)\right)\\
&\simeq& \colim\left(F \circ \alpha^{op}\right).
\end{eqnarray*}
\end{proof}

\begin{rmk}
Strictly speaking, we do not need to appeal to Theorem \ref{thm:embeql}, as it can be shown quite easily that $\alpha^*$ induces an equivalence between the $\i$-categories of \emph{hyper}sheaves, and one can then appeal to the fact that $\Shi\left(\nMfd^{\et}\right)$ is hypercomplete (\cite{higherme}, Theorem 5.3.6), and then replace the role of $a$ with the hyper-sheafification functor. However, we do not find this as elegant, and we find Theorem \ref{thm:embeql} interesting in its own right.
\end{rmk}

At this point, we wish to remind the reader of a certain categorical construction:

\begin{dfn}\label{dfn:Groth1}
Let $F:\C \to \mathbf{Cat}$ be a functor from a small category $\C$ to the category of small categories. The \textbf{Grothendieck construction} of $F$ is the category
$$\int_\C F$$ whose objects are pairs $\left(C,X\right)$ with $C$ an object of $\C$ and $X$ an object in $F\left(C\right).$ A morphism from $\left(C,X\right)$ to $\left(D,Y\right)$ is a morphism $f:C \to D$ in $\C$ together with a morphism $$g:F\left(f\right)\left(X\right) \to Y$$ in $F\left(D\right).$ Note that there is a canonical functor $$\pi_F:\int_\C F \to \C.$$
There is also a contravariant version, namely if $F:\C^{op} \to \mathbf{Cat}$, we can apply the Grothendieck construction to get a functor $$\int_{\C^{op}} F \to \C^{op},$$ and then consider (with slight abuse of notation) the induced functor $$\pi_F:\left(\int_\C F\right):=\left(\int_{\C^{op}} F\right)^{op} \to \left(\C^{op}\right)^{op}=\C,$$ which we will also call the Grothendieck construction of $F$.
\end{dfn}

\begin{rmk}\label{rmk:Groth1}
There is a slight generalization of the above construction  which works for weak functors $$F:\C \to \mathbf{Cat},$$ rather than than strict functors, that is one which works for morphisms of bicategories, where $\C$ is regarded as a bicategory whose only $2$-morphisms are identities, and $\mathbf{Cat}$ is the bicategory of small categories, functors, and natural transformations. The objects and arrows of the category $$\int_\C F$$ are the same as in Definition \ref{dfn:Groth1}, however the coherency data for $F$ is needed to define composition in this category.
\end{rmk}

We now recall a classical theorem of Thomason:

\begin{thm} (Theorem 1.2 of \cite{Thomason})\\
Let $F:\C \to \mathbf{Cat}$ be a functor from a small category $\C$ to the category of small categories, and consider the composite $$N \circ F: \C \to \Set^{\Delta^{op}}.$$ Then there is a natural homotopy equivalence of simplicial sets $$\hocolim \left(N \circ F\right) \stackrel{\sim}{\longrightarrow} N\left( \int_{\C} F\right),$$ where $$\pi_F:\int_{\C} F \to \C$$ is the Grothendieck construction of $F$.
\end{thm}

The following is a reformulation in the language of $\i$-categories:
\begin{cor} \label{cor:thom2}
Let $F:\C \to \mathbf{Cat}$ be a functor from a small category $\C$ to the category of small categories, and consider the composite $$\C  \stackrel{F}{\longrightarrow} \mathbf{Cat} \stackrel{B}{\longrightarrow} \Top \stackrel{h}{\longrightarrow} \iGpd.$$ Then the colimit of the above composite is $$h\left(B\left(\int_{\C} F\right)\right).$$ In particular, if $G:\C \to \Set$ is any functor, then the colimit of the composite $$\C  \stackrel{G}{\longrightarrow} \Set \hookrightarrow \iGpd.$$ can be identified with $$h\left(B\left(\int_{\C} G\right)\right).$$
\end{cor}

As an immediate consequence of Corollary \ref{cor:ex1}, Theorem \ref{thm:main} and Corollary \ref{cor:thom2} we get another proof of a celebrated theorem of Segal:

\begin{thm} \label{thm:segal}(Proposition 1.3 of \cite{Segal})\\
Let $\Gamma^n$ be the $n$-dimensional Haefliger groupoid of Example \ref{ex:Haefliger}. Then there is a weak homotopy equivalence between classifying spaces:
$$B\left(\Emb\right) \to B \Gamma^n.$$ 
\end{thm}

\begin{proof}
Consider the $n^{th}$ Haefliger stack $\HA_n=\left[\Gamma^n\right].$ On one hand, by Corollary \ref{cor:diffhom}, we have that $$\Pi_\i\left(\HA_n\right) \simeq h\left(B \Gamma^n\right).$$ On the other hand, by Corollary \ref{cor:ex1}, we have that $$\HA_n \simeq j^n_!\left(1\right),$$ where $1$ is the terminal object of $\Shi\left(\nMfd^{\et}\right).$ By Theorem \ref{thm:main}, it follows that $\Pi_\i\left(j^n_!\left(1\right)\right)$ is the colimit of the constant functor $$\Emb^{op} \to \iGpd$$ with value the terminal object. By Corollary \ref{cor:thom2}, we have that this colimit may be expressed as the $\i$-groupoid associated to the classifying space of $$\int\limits_{\Emb^{op}}\!\!\!\!\!\!\!\!1 \cong \Emb^{op}.$$ For any small category $\C,$ there is a canonical homeomorphism $$B \C \cong B \C^{op},$$ so it follows that $\Pi_\i\left(j^n_!\left(1\right)\right)$ is equivalent to $h\left(B\left(\Emb\right)\right),$ and we conclude that $h\left(B \Gamma^n\right)$ and $h\left(B\left(\Emb\right)\right)$ are equivalent in the $\i$-category of $\i$-groupoids, so in particular, $B\left(\Emb\right)$ and $B \Gamma^n$ are isomorphic in the homotopy category of spaces. Since $B\left(\Emb\right)$ is the geometric realization of a simplicial set, it is a CW-complex, and hence cofibrant, so there must exist a weak homotopy equivalence $$B\left(\Emb\right) \to B\Gamma^n$$ exhibiting the desired equivalence.
\end{proof}

\begin{thm}\label{thm:segalsymp}
Let $Sp_{2n}$ denote the following category. The objects consist of symplectic forms $\omega$ on $\RR^{2n}$. An arrow $\omega \to \omega'$ between two such symplectic forms is an embedding $$\varphi:\R^{2n} \hookrightarrow \R^{2n}$$ such that $\varphi^*\omega'=\omega.$ Then there is a weak homotopy equivalence between classifying spaces:
$$B\left(Sp_{2n}\right) \to B \Gamma^{Sp}_{2n},$$ where $\Gamma^{Sp}_{2n}$ is the Lie groupoid from Example \ref{ex:sympgpd}.
\end{thm}

\begin{proof}
The category $Sp_{2n}$ is easily seen to be the Grothendieck construction of the functor $\mathcal{S}_{2n}|_{\operatorname{Emb}\left(\RR^{2n}\right)},$ where $\mathcal{S}_{2n}$ is the functor on $2n\mbox{-}\Mfd^{\et}$ sending every $2n$-manifold $M$ to its set of symplectic forms. By Example \ref{ex:symprol}, we have that $$j^{2n}_!\left(\mathcal{S}_{2n}\right)\simeq \left[\Gamma^{Sp}_{2n}\right].$$ By Theorem \ref{thm:main}, it follows that $\Pi_\i\left(j^{2n}_!\left(\mathcal{S}_{2n}\right)\right)$ is the colimit of the functor $$\mathcal{S}_{2n}|_{\operatorname{Emb}\left(\RR^{2n}\right)}:\operatorname{Emb}\left(\RR^{2n}\right)^{op} \to \iGpd.$$ By Corollary \ref{cor:thom2}, we have that this colimit may be expressed as the $\i$-groupoid associated to the classifying space of the Grothendieck construction of $\mathcal{S}_{2n}|_{\operatorname{Emb}\left(\RR^{2n}\right)}$, i.e. as $h\left(B\left(Sp_{2n}\right)\right).$ The rest of the proof is analogous to that of Theorem \ref{thm:segal}.
\end{proof}

\begin{thm}\label{thm:segalriem}
Let $\cR \mbox{iem}_n$ denote the following category. The objects consist of Riemannian metrics $g$ on $\RR^{n}$. An arrow $g \to g'$ between two such metrics is an embedding $$\varphi:\R^{n} \hookrightarrow \R^{n}$$ such that $\varphi^*g'=g.$ Then there is a weak homotopy equivalence between classifying spaces:
$$B\left(\cR \mbox{iem}_n\right) \to B \cR\Gamma^{n},$$ where $\cR\Gamma^{n}$ is the Lie groupoid from Example \ref{ex:riemgpd}.
\end{thm}

\begin{proof}
The proof is completely analogous to that of Theorem \ref{thm:segalsymp} with the role of Example \ref{ex:symprol} played by Example \ref{ex:riemprol}.
\end{proof}

We have the following generalization of Thomason's theorem proven by Lurie:

\begin{thm} \label{thm:thom3} (Corollary 3.3.4.6 of \cite{htt})\\
Let $F:\C \to \iGpd$ be a functor from a small $\i$-category $\C$ to the $\i$-category of $\i$-groupoids. Denote by $$\pi_F:\int_\C F \to \C$$ the left fibration classified by $F$ (the $\i$-analogue of the Grothendieck construction). Then the colimit of $F$ can be identified with $$h\left(\left|\mspace{4mu}\int_\C F \mspace{4mu}\right|\right)$$ where we have regarded $\int_\C F$ as a quasi-category, hence a simplicial set.
\end{thm}

\begin{cor}\label{cor:thom}
If $F:\C \to \Gpd$ is a weak functor from a small category $\C$ to the bicategory of groupoids, functors, and natural transformations,
then the homotopy colimit of the composite $$\C \stackrel{F}{\longrightarrow} \Gpd \hookrightarrow \iGpd$$ is $$h\left(B\left(\int_{\C} F\right)\right).$$
\end{cor}

\begin{proof}
By Prop 2.1.1.3, the left fibration classified by the above composite is the nerve of the Grothendieck construction: $$N\left(\int_\C F\right) \to N\left(\C\right).$$ The result now follows from Theorem \ref{thm:thom3}.
\end{proof}

\begin{cor}\label{cor:thomasonlurie}
Suppose that $F$ in $\Shi\left(\nMfd^{\et}\right)$ is an $\i$-sheaf on $n$-manifolds and their local diffeomorphisms, and let $\X=j^n_!F$ be its associated $n$-dimensional \'etale differentiable $\i$-stack. Then  $\Pi_\i\left(\X\right)$ is equivalent to the $\i$-groupoid associated to the geometric realization of the underlying simplicial set of $$\int\limits_{\Emb}\!\!\!\!\!\!\!\!\alpha^*F,$$ i.e. 
$$h\left(\left|\mspace{4mu}\int\limits_{\Emb}\!\!\!\!\!\!\!\!\alpha^*F\mspace{4mu}\right|\right).$$
\end{cor}

\begin{proof}
This is an immediate corollary of Theorem \ref{thm:main} and Corollary \ref{cor:thomasonlurie}.
\end{proof}

\begin{cor}\label{cor:thomasonlurie2}
Suppose that $F$ in $\St\left(\nMfd^{\et}\right)$ is a stack of groupoids on $n$-manifolds and their local diffeomorphisms, and let $\X=j^n_!F$ be its associated $n$-dimensional \'etale differentiable stack. Then  $\Pi_\i\left(\X\right)$ is equivalent to the $\i$-groupoid associated to the classifying space of the Grothendieck construction $$\int\limits_{\Emb}\!\!\!\!\!\!\!\!\alpha^*F,$$ i.e. 
$$h\left(B\left(\int\limits_{\Emb}\!\!\!\!\!\!\!\!\alpha^*F\mspace{4mu}\right)\right).$$
\end{cor}

\begin{proof}
This follows immediately from Corollary \ref{cor:thomasonlurie} and Corollary \ref{cor:thom}.
\end{proof}

\section{The Homotopy Type of Almost Free Global Quotients}\label{sec: Borel}

If $G \acts T$ is a continuous action of a topological group $G$ on a topological space $T,$ then we can model the homotopy quotient $T_G$ by the Borel construction:

\begin{dfn}
Let $BG$ be a classifying space for $G$ (e.g. $BG=||N\left(G\right)||$) and let $EG \to BG$ be a universal principal $G$-bundle (e.g. $EG=||N\left(G \ltimes G\right)||$), then the \textbf{Borel construction} of $G \acts T$ is the quotient of $X \times EG$ by the diagonal action of $G$, and is denoted by $X \times_G EG$.
\end{dfn}

\begin{rmk}
In general, the Borel construction, as a topological space, depends on the choice of classifying space and universal bundle, however as a weak homotopy type, it is well defined.
\end{rmk}

Going to the smooth setting, let $G \acts M$ be a smooth almost free action of a Lie group. (Recall that an almost free action is one for which each stabilizer group $G_x$ is discrete). Then by Example \ref{ex:action} and Proposition \ref{prop:folgpd}, we have that $$M//G:=\left[G\ltimes M\right]$$ is an \'etale differentiable stack of dimension $$\dim M//G= \dim M - \dim G.$$ 

\begin{rmk}
If $G$ is compact, then $M//G$ is moreover an orbifold. In fact, every effective (aka reduced) orbifold $\X$ is of this form, where $M$ can be taken to be the orthonormal frame bundle with respect to a Riemannian metric $g$ on $\X$ (this frame bundle turns out to be a manifold) and where $G$ can be taken to be $O\left(n\right),$ where $n$ is the dimension of $\X$ \cite{satake}.
\end{rmk}

The following result was proven by Schrieber for arbitrary Lie group actions:

\begin{prop}
For $G$ a Lie group acting smoothly on a manifold $M,$ one has $$\Pi_\i\left(M//G\right) \simeq h\left(M \times_G EG\right),$$ where $M \times_G EG$ is the Borel construction.
\end{prop}

\begin{proof}
This follows from Propositions 4.3.32 and 4.4.13 of \cite{dcct}.
\end{proof}

In other words, the weak homotopy type of $M//G$ agrees with the homotopy quotient of $M$ by $G$. In this section, we will show that when the action is almost free, there is another natural description of the homotopy quotient which is more geometric in nature. More precisely, we will construct a natural category of geometric data whose classifying space has the same homotopy type as $M \times_G EG.$

The category is as follows:

\begin{dfn}\label{dfn:Borel}
Let $G$ be a Lie group of dimension $k$ and let $M$ be a smooth $\left(n+k\right)$-manifold equipped with an almost free $G$-action. Define the \textbf{Borel category} of $G \acts M$, denoted as $\Bor$, to be the following category. The objects consist of smooth maps $$f:\R^n \to M$$ which are transverse to all the $G$-orbits. An arrow from $f$ to $f'$ consists of a pair $\left(\varphi,\tau\right)$ with $$\varphi:\R^n \hookrightarrow \R^n$$ an embedding, and with $$\tau:\R^n \to G$$ a smooth map such that for all $x \in \R^n$ we have $$f\left(x\right)=\tau\left(x\right) \cdot f'\left(\varphi\left(x\right)\right).$$ Composition $$f \stackrel{\left(\varphi,\tau\right)}{\longlonglongrightarrow} f' \stackrel{\left(\varphi',\tau'\right)}{\longlonglongrightarrow} f''$$ is defined as follows:
$$\left(\varphi',\tau'\right) \circ \left(\varphi,\tau\right):=\left(\varphi'\varphi,c\left(\tau,\tau',\varphi\right)\right)$$ where 
\begin{eqnarray*}
c\left(\tau,\tau',\varphi\right):\R^n &\to& G\\
x &\mapsto& \tau\left(x\right) \cdot \tau'\left(\varphi\left(x\right)\right).
\end{eqnarray*}
\end{dfn}

The main goal of this section is to prove the following theorem:

\begin{thm}\label{thm:Borel}
Let $G$ be a Lie group of dimension $k$ and let $M$ be a smooth $\left(n+k\right)$-manifold equipped with an almost free $G$-action. Then $$\Pi_\i\left(M//G\right) \simeq h\left(B\left(\Bor\right)\right).$$ In particular, $B\left(\Bor\right)$ has the same weak homotopy type as $M \times_G EG.$
\end{thm}

The main idea is as follows: The \'etale differentiable stack $M//G$ is $n$-dimensional, so there exists a stack $\Z$ on the site $\nMfd^{\et}$ of smooth $n$-manifolds and their local diffeomorphisms, such that $$j^n_!\left(\Z\right) \simeq M//G.$$ In fact, by Theorem \ref{thm:6.1.3}, $\Z$ is the stack which assigns to each $n$-manifold $N$ the groupoid of local diffeomorphisms $N \to M//G$. After getting a good understanding of this stack, we can apply Corollary \ref{cor:thomasonlurie} to compute its homotopy type.

\begin{prop}\label{prop:Z}
Let $\Z$ be the stack which assigns to each $n$-manifold $N$ the groupoid of local diffeomorphisms $N \to M//G,$ then it can be described concretely as follows: The objects of $\Z\left(N\right)$ are principal $G$-bundles $P \to N$ together with an equivariant local diffeomorphism $P \to M.$ The arrows are maps of principal $G$-bundles over $N$ which commute over $M$.
\end{prop}

\begin{proof}
It is standard that $M//G\left(N\right)$ can be described as the groupoid of principal $G$-bundles $P \to N$ together with an equivariant map $P \to M,$ but we will explain this fact for the reader's convenience. Firstly, notice that the source map of the action groupoid $G \ltimes M$ is simply projection onto $M,$ so one sees directly that a $G\ltimes M$-manifold $E$ (Definition \ref{def:Gmfd}) is the same as a $G$-manifold equipped with an equivariant map $E \to M,$ where the equivariant map in question is the moment map. By a similar argument, one sees that a principal $G \ltimes M$-bundle over $N$ (Definition \ref{dfn:Gbundle}) is the same as a principal $G$-bundle $P \to N$ equipped with a $G$-equivariant map $P \to M.$ Finally, Corollary \ref{cor:buncon} tells us that the groupoid of such principal $G \ltimes M$-bundles over $N$ is equivalent to $M//G\left(N\right).$ By the same corollary, if $f:N \to M//G$, then we have a pullback diagram
$$\xymatrix{P \ar[r]^-{\mu}  \ar[d]_-{\pi} & M \ar[d]\\
N \ar[r]^-{f} & M//G,}$$
 so if $f$ is a local diffeomorphism, so is $\mu$. Conversely, if $\mu$ is a local diffeomorphism, observe that the morphism $f$ is induced by a morphism of Lie groupoids $$Pair\left(\pi\right) \to G\ltimes M$$ (see Example \ref{ex:pair}) which is $\mu$ on objects, and hence $f$ is also a local diffeomorphism, by Remark \ref{rmk:localdiffeo}.
\end{proof}

\begin{lem}\label{lem:submersion}
Suppose that $\rho:G \times M \to M$ is a smooth action of a Lie group and that $f:N \to M$ is a smooth map. Then the composite $$G \times N \stackrel{id \times f}{\longlongrightarrow} G \times M \stackrel{\rho}{\longrightarrow} M$$ is a submersion if and only if $f$ is transverse to all the $G$-orbits of $M$.
\end{lem}

\begin{proof}
Note that any tangent vector $V \in T_g G$ is of the form $g \cdot X=T_e\left(L_g\right)\left(X\right)$ for a unique $X \in \mathfrak{g},$ where $$L_g:G \stackrel{\cong}{\longrightarrow} G$$ is left-multiplication. Consider the smooth map
$$\xymatrix@R=0.1in@C=1.2cm{G \times M \ar[r]^-{L_g \times id} & G \times M \ar[r]^-{\rho} & M\\
\left(e,x\right) \ar@{|-{>}}[r] & \left(g,x\right) \ar@{|-{>}}[r] & g\cdot x.}$$
By the chain rule we have 

\begin{equation}\label{eq:star}
\resizebox{4.6in}{!}{$T_{\left(e,x\right)}\left(\rho \circ \left(L_g \times id\right)\right) \left(X,0\right)=T_{\left(g,x\right)}\left(\rho\right)\left(T_{\left(e,x\right)}\left(L_g \times id\right)\left(X,0\right)\right)=T_{\left(g,x\right)}\left(\rho\right)\left(\left(gX,0\right)\right).$
}
\end{equation}
Consider the curve $$\gamma\left(t\right)=\left(e^{tX},x\right)$$ through $\left(e,x\right)$ in $G \times M,$ where $e^{tX}$ is the exponential map of the Lie group applied to $X$. The tangent vector of this curve at $\left(e,x\right)$ is $\left(X,0\right).$ Notice that we have that
\begin{eqnarray*}
\rho\circ \left(L_g \times id\right)\left(\gamma\left(t\right)\right)&=&g \cdot e^{tX} \cdot x\\
&=&\left(ge^{tX}g^{-1}\right) \cdot gx\\
&=& e^{t Ad\left(g\right)\left(X\right)} \cdot gx,
\end{eqnarray*}
where $Ad$ denotes the adjoint action of $G$ on $\mathfrak{g}$,
from which it follows that $$\frac{d}{dt}\left(\rho\circ \left(L_g \times id\right)\left(\gamma\left(t\right)\right)\right)=\widehat{Ad\left(g\right)\left(X\right)}_{gx}=T_{\left(e,x\right)}\left(\rho \circ\left(L_g  \times id\right)\right)\left(X,0\right),$$ where we have used the induced infinitesimal action of the Lie algebra
\begin{eqnarray*}
\rho_*:\mathfrak{g} &\to& \mathfrak{X}\left(M\right)\\
Y &\mapsto& \widehat{Y}.
\end{eqnarray*}
From equation (\ref{eq:star}), it follows that
$$T_{\left(g,x\right)}\left(\rho\right)\left(gX,0\right) = \widehat{Ad\left(g\right)\left(X\right)}_{gx}.$$
Now let $Y \in T_x\left(M\right)$, and suppose that $s\left(t\right)$ is a curve in $M$ such that $s\left(0\right)=x$ and whose tangent vector at $x$ is $Y.$ Let $$\tilde s\left(t\right)=\left(g,s\left(t\right)\right).$$ Then $$\rho\left(\tilde s\left(t\right)\right)=g\cdot s\left(t\right)$$ so it follows that $$T_{\left(g,x\right)}\left(\rho\right)\left(0,Y\right)=T_x\left(\rho_g\right)\left(Y\right),$$ where $$\rho_g=\rho\left(g,\mspace{3mu} \cdot \mspace{3mu}\right):M \stackrel{\cong}{\longlongrightarrow} M.$$ So by the chain rule and linearity we have that for $p \in N,$ $Z \in T_pN$ and $X \in \mathfrak{g}$ that:
$$T_{\left(g,p\right)}\left(\rho\circ \left(id \times f\right)\right)\left(gX,Z\right)=\widehat{Ad\left(g\right)\left(X\right)}_{gf\left(p\right)}+T_{f\left(p\right)}\left(\rho_g\right)\left(T_p\left(f\right)\left(Z\right)\right).$$
Now for any point $x \in M,$ if $\O_x$ denotes its $G$-orbit, then we have
$$T_x\left(\O_x\right) = \left\{\widehat{X}_x\mspace{3mu}|\mspace{3mu}X \in \mathfrak{g}\right\},$$ and since $$Ad\left(g\right):\mathfrak{g} \to \mathfrak{g}$$ is an isomorphism of vector spaces, it follows that

\begin{equation}\label{eq:image}
\operatorname{Im}\left(T_{\left(g,p\right)}\left(\rho\circ \left(id \times f\right)\right)\right)= T_{gf\left(p\right)}\left(\O_{f\left(p\right)}\right)+T_{f\left(p\right)}\left(\rho_g\right)\left(\operatorname{Im}\left(T_p\left(f\right)\right)\right).
\end{equation}
When $g=e$ this simplifies to
$$\operatorname{Im}\left(T_{\left(e,p\right)}\left(\rho\circ \left(id \times f\right)\right)\right)= T_{f\left(p\right)}\left(\O_{f\left(p\right)}\right)+\operatorname{Im}\left(T_p\left(f\right)\right).$$ So by definition, we see that $f$ is transverse to all the $G$-orbits if and only if $\rho\circ \left(id \times f\right)$ is a submersion at all points of the form $\left(e,x\right).$

Clearly, if $\rho\circ \left(id \times f\right)$ is a submersion, then it is a submersion at $\left(e,x\right)$ for all $x.$ We will complete our proof by showing the converse. Suppose that $\rho\circ \left(id \times f\right)$ is a submersion at $\left(e,x\right).$ We will show that $\rho\circ \left(id \times f\right)$ is also a submersion at $\left(g,x\right).$ Let $W \in T_{gf\left(p\right)}M,$ then we have that $$\rho_g:M \stackrel{\cong}{\longlongrightarrow} M,$$ so that $$T_{f\left(p\right)}\left(\rho_g\right):T_{f\left(p\right)}M \stackrel{\cong}{\longlongrightarrow} T_{gf\left(p\right)}M,$$ and hence $$W=T_{f\left(p\right)}\left(\rho_g\right)\left(V\right),$$ for a unique $V \in T_{f\left(p\right)}M=T_{e\cdot f\left(p\right)}M.$ Since $\rho\circ \left(id \times f\right)$ is a submersion at $\left(e,x\right),$ this $V$ can be expressed as
$$V=Z+T_p\left(f\right)\left(Y\right)$$ with $Y \in T_p N$ and $Z \in T_{f\left(p\right)} \O_{f\left(p\right)},$ and so $$W=T_{f\left(p\right)}\left(\rho_g\right)\left(Z\right)+T_{f\left(p\right)}\left(\rho_g\right)\left(\left(T_p\left(f\right)\left(Y\right)\right)\right).$$ Notice that since $\rho_g$ preserves $G$-orbits, we must have $$T_{f\left(p\right)}\left(\rho_g\right)\left(Z\right) \in T_{gf\left(p\right)} \O_{f\left(p\right)}$$ and hence by (\ref{eq:image}), $W$ is in the image of $T_{\left(g,p\right)}\left(\rho\circ \left(id \times f\right)\right).$
\end{proof}

Proof of Theorem \ref{thm:Borel}:

\begin{proof}
Let $\Z$ be the stack of groupoids from Proposition \ref{prop:Z}. Then by Theorem \ref{thm:6.1.3}, $j^n_!\Z \simeq M//G$ and so by  Corollary \ref{cor:thomasonlurie2}, we have that $$\Pi_\i\left(M//G\right) \simeq h\left(B\left(\int\limits_{\Emb}\!\!\!\!\!\!\!\!\alpha^*\Z\mspace{4mu}\right)\right).$$ We will now show that we have an equivalence of categories $$\Bor \simeq \int\limits_{\Emb}\!\!\!\!\!\!\!\!\alpha^*\Z.$$ By Proposition \ref{prop:Z}, it follows that the above Grothendieck construction can be described as the category whose objects are pairs $\left(P,\psi\right)$ with $P \to \RR^n$ a principal $G$-bundle and $\psi:P \to M$ a $G$-equivariant local diffeomorphism. A morphism $$\left(P,\psi\right) \to \left(P',\psi'\right)$$ between such pairs is a pair $\left(\varphi,\theta\right)$ consisting of an embedding $$\varphi:\RR^n \hookrightarrow \RR^n$$ and an isomorphism $$\theta:P \to \varphi^*P'$$ of principal $G$-bundles over $\RR^n$ such that the composite $$P \stackrel{\theta}{\longlongrightarrow} \varphi^*P' \to P' \stackrel{\psi'}{\longlongrightarrow} M$$ is $\psi$. Note that a more natural description of the morphisms $$\left(P,\psi\right) \to \left(P',\psi'\right)$$ is that such a morphism is pair $\left(\varphi,\sigma\right),$ where $$\varphi:\RR^n \hookrightarrow \RR^n$$ is an embedding and $\sigma$ is a
 morphism of principal $G$-bundles over $\varphi$ such that the following diagram commutes:
\begin{equation}\label{eq:principal}
\xymatrix{
P \ar[d]  \ar@/^2pc/[rr]^-{\psi} \ar[r]^-{\sigma} & P' \ar[d] \ar[r]^-{\psi'} & M\\
\RR^n \ar@{^{(}->}[r]^-{\varphi} & \RR^n. & }
\end{equation}

Every principal $G$-bundle over $\RR^n$ is trivializable, so for every pair $\left(P,\psi\right),$ we can find $$\theta:G \times \RR^n \stackrel{\cong}{\longlongrightarrow} P$$ an isomorphism of principal $G$-bundles over $\RR^n,$ and hence we have an isomorphism $$\left(id,\theta\right):\left(G \times \RR^n,\psi \circ \theta\right)\stackrel{\cong}{\longlongrightarrow} \left(P,\psi\right).$$ It follows that we can restrict to the full subcategory on those objects of the form $\left(G \times \RR^n,\psi\right),$ i.e. those objects $\left(P,\psi\right),$ with $P=G\times \RR^n$ the trivial principal $G$-bundle, and end up with an equivalent category. Lets denote this full subcategory of by $\sD$.

Observe that there is natural bijection between smooth maps $f:\RR^n \to M$ and $G$-equivariant maps $G \times \RR^n \to M$ given by the assignment $$f \mapsto G \times \RR^n \stackrel{id \times f}{\longlongrightarrow} G \times M \stackrel{\rho}{\longrightarrow} M,$$ where $\rho:G \times M \to M$ is the $G$-action. Since $\dim\left(G \times \RR^n\right)=\dim\left(M\right),$ the composite $\rho \circ \left(id \times f\right)$ is a local diffeomorphism if and only if it is a submersion, so it follows from Lemma \ref{lem:submersion} that the above bijection restricts to a bijection between smooth maps $f:M \to \RR^n$ which are transverse to all the $G$-orbits and $G$-equivariant local diffeomorphisms $G \times \RR^n \to M.$ Hence the objects of $\sD$ can be taken to be morphisms $f:\RR^n \to M$ which are transverse to all the $G$-orbits. For such a smooth map $f$, denote by $\tilde f$ the associated $G$-equivariant local diffeomorphism $\rho \circ \left(id \times f\right).$

Observe furthermore that there is a natural identification $$C^\infty\left(\RR^n,G\right)\cong \Hom_{\varphi}\left(G\times\RR^n,G\times\RR^n\right),$$ where the right hand side is the set of morphisms of principal $G$-bundles which induce the embedding $\varphi$ on the base $\RR^n,$
given by the assignment
\begin{eqnarray*}
C^\infty\left(\RR^n,G\right) &\to& \Hom_{\varphi}\left(G\times\RR^n,G\times\RR^n\right)\\
\tau &\mapsto& \tilde \tau:\left(g,x\right) \mapsto \left(g\tau\left(x\right),\varphi\left(x\right)\right).
\end{eqnarray*}
If $f$ and $f'$ are smooth maps $\RR^n \to M$ which are transverse to the $G$-orbits, and $\tilde f$ and $\tilde f'$ are the induced $G$-equivariant local diffeomorphisms $$G \times \RR^n \to M,$$ for a fixed $\varphi:\RR^n \hookrightarrow \RR^n,$ a given $\tau:\RR^n \to G$ has the property that the pair $\left(\varphi,\tilde \tau\right)$ is a morphism $\left(G\times\RR^n,\tilde f\right) \to \left(G \times \RR^n, \tilde f'\right)$ in $\sD$ (as in (\ref{eq:principal})) if and only if, for all $x \in \RR^n,$ 
\begin{equation}\label{eq:condition}
\tau\left(x\right) f'\left(\varphi\left(x\right)\right) = f\left(x\right).
\end{equation}
So we have seen that we can identify the objects of $\sD$ with smooth maps $$f:\RR^n \to M$$ which are transverse to all the $G$-orbits, and can identify morphisms $f \to f'$ between two smooth maps as pairs $\left(\varphi,\tau\right)$ with $\varphi:\RR^n \hookrightarrow \RR^n$ a smooth embedding and $\tau:\RR^n \to G$ such that for all $x \in \RR^n,$ equation (\ref{eq:condition}) holds. Suppose now that we have two composable pairs of arrows in $\sD$: $$f \stackrel{\left(\varphi,\tilde \tau\right)}{\longlonglongrightarrow} f' \stackrel{\left(\varphi',\tilde \tau'\right)}{\longlonglongrightarrow} f''.$$ Then since we have
$$\xymatrix@R=0.1in@C=1.2cm{G \times \RR^n \ar[r]^-{\tilde \tau} & G \times \RR^n \ar[r]^-{\tilde \tau'} & G \times \RR^n\\
\left(g,x\right) \ar@{|-{>}}[r] & \left(g\tau\left(x\right),\varphi\left(x\right)\right) \ar@{|-{>}}[r] & \left(g\tau\left(x\right)\tau'\left(\varphi'\left(x\right)\right),\varphi'\varphi\left(x\right)\right),}$$
it follows that $$\left(\varphi',\tilde \tau'\right) \circ \left(\varphi,\tilde \tau\right)=\left(\varphi'\varphi,\widetilde{c\left(\tau,\tau',\varphi\right)}\right),$$ where $$c\left(\tau,\tau',\varphi\right)\left(x\right)=\tau\left(x\right)\tau'\left(\varphi\left(x\right)\right).$$ From Definition \ref{dfn:Borel}, it follows that $\sD$ is canonically isomorphic to the Borel category $\Bor$. Since the canonical inclusion $$\sD \hookrightarrow \int\limits_{\Emb}\!\!\!\!\!\!\!\!\alpha^*\Z$$ is an equivalence of categories, it follows that there is an induced homotopy equivalence $$B\left(\Bor\right) \to B\left(\int\limits_{\Emb}\!\!\!\!\!\!\!\!\alpha^*\Z\right),$$ so finally
$$h\left(B\left(\Bor\right)\right) \stackrel{\sim}{\longlongrightarrow} h\left(B\left(\int\limits_{\Emb}\!\!\!\!\!\!\!\!\alpha^*\Z\right)\right) \simeq \Pi_\i\left(M//G\right).$$
\end{proof}

\begin{rmk}
If a particular model for the Borel construction $M \times_{G} EG$ has the homotopy type of a $CW$-complex, which is the case e.g. if $M$ is $2^{nd}$ countable and Hausdorff and one uses a model for $EG$ as a $G$-CW-complex, or if one uses $||N\left(G \times M\right)||$ for $M \times_{G} EG,$ then it follows by Whitehead's theorem that there is in fact a homotopy equivalence $$B\left(\Bor\right) \to M \times_{G} EG.$$
\end{rmk}




\appendix

\section{A Brief Introduction to $\i$-Categories.}\label{sec:infinity}
In this appendix, we will try to provide enough information (some of which will heuristic) for the reader unfamiliar with higher category theory to read the main body of this article. Those readers who already have a background in higher category theory are encouraged to only glance through this appendix to become familiar with notational conventions.

\subsection{$\i$-Groupoids and the Homotopy Hypothesis}
A groupoid is a category in which every arrow is an isomorphism. As a special case, any group $G$ determines a groupoid with precisely one object $x$ such that $\Hom\left(x,x\right)=G.$ One categorical dimension higher, a $2$-groupoid is a weak $2$-category, or bicategory, in which every $2$-morphism is an isomorphism and every $1$-morphism is an \emph{equivalence}. Recall that an arrow $f:x \to y$ in a bicategory is an equivalence if there exists another arrow $g: y \to f$ and invertible $2$-morphisms $\alpha:f \circ g \Rightarrow id_y$ and $\beta:g \circ f \Rightarrow id_x.$ (Notice the similarity between the notion of an equivalence and the notion of a homotopy equivalence). 


For simplicity, let us assume that all topological spaces we speak of have the homotopy type of a $CW$-complex. Given a topological space $X,$ one has associated to it a canonical groupoid- its \emph{fundamental groupoid} $\Pi_1X$. The objects of $\Pi_1X$ consist of the points of $X$ and the arrows are given by homotopy classes of paths between points. If $X$ does not have any higher homotopy groups besides its fundamental group- making $X$ a so-called \emph{$1$-type}- then the groupoid $\Pi_1X$ contains all the homotopical information about $X.$ More precisely, the classifying space $B \Pi_1 X$ is naturally homotopy equivalent to $X$. More generally, if $\G$ is any groupoid, its classifying space $B\G$ has no higher homotopy groups besides its fundamental group. A special case of this is if $G$ is a group, then $BG$ is a $K\left(G,1\right)$, and we see that $G$ may be recovered as the fundamental group of $BG$. This generalizes to the fact that the fundamental groupoid of $B\G$ is  equivalent to $\G$. These facts can be used to establish an equivalence of bicategories between, on one hand the bicategory of (small) groupoids, functors, and natural transformations, and on the other hand, the bicategory of $1$-types, continuous maps, and homotopy classes of homotopies.  (Notice that two $1$-types are equivalent in this bicategory if they are homotopy equivalent.) Roughly speaking, one has that groupoids and $1$-types are the same thing.

A \emph{$2$-type} is a topological space $X$ such that $\pi_i\left(X\right)=0$ for all $i >2.$ A similar construction yields an equivalence between the tricategory of (weak) $2$-groupoids and the tricategory of $2$-types, continuous maps, homotopies, and homotopy classes of homotopies between homotopies. Roughly speaking, $2$-groupoids are the same as $2$-types.

A \emph{$n$-type} is a topological space $X$ such that $\pi_i\left(X\right)=0$ for all $i >n.$ The \textbf{Homotopy Hypothesis}, roughly speaking, says that for \emph{any} $n$, (weak) $n$-groupoids and $n$-types are the same thing. (Depending on the model of $n$-groupoids used, the Homotopy Hypothesis is either a theorem, or is true by definition.) The Homotopy Hypothesis holds in particular for $n=\infty$ so that an $\i$-groupoid is the same thing as an arbitrary homotopy type. The idea behind this is as follows: An $\infty$-groupoid should have objects, arrows, $2$-arrows, $3$-arrows, ad infinitum, and all the arrows, $2$-arrows, $3$-arrows etc. should be equivalences. This means for instance, that the arrows should be invertible up to $2$-morphisms which are invertible up to $3$-morphisms, which are invertible up to $4$-morphisms... Given a topological space $X$, one can construct such an $\infty$-groupoid- its fundamental $\infty$-groupoid $\Pi_\infty X$. The objects are the points of $x,$ the arrows are paths between points, the $2$-arrows are homotopies between paths, the $3$-arrows are homotopies between homotopies and so on. The Homotopy Hypothesis states that, up to equivalence, all $\infty$-groupoids arise in this way.

\subsection{$\infty$-Categories}
Roughly speaking, an $\infty$-category is a higher category $\C$ which has objects, arrows between these objects, $2$-arrows between these arrows, $3$-arrows between these $2$-arrows etc. such that all $k$-arrows for $k>1$ are equivalences.\footnote{More precisely, this is what is called an $\left(\infty,1\right)$-category. A general $\left(\infty,n\right)$-category satisfies the same condition for $k > n.$} Another way of saying this is that for any two objects $X$ and $Y$, $\Hom_\C\left(X,Y\right)$ should be an $\infty$-groupoid. 

Invoking the homotopy hypothesis, one way of describing $\infty$-categories is by categories enriched in spaces. (However there are many equivalent ways of modeling $\infty$-categories, see for instance \cite{Julie,BarwickKan} and Chapter 1.1 of \cite{htt}). In other words, one can model an $\icat$ simply by a category $\C$ such that for each pair of objects $X$ and $Y$ in $\C,$ the set $\Hom\left(X,Y\right)$ has the structure of a topological space, in such a way that all the structure maps of the category $\C$ (for instance composition) are continuous maps of these spaces of arrows. Notice that to such an enriched category $\C,$ by applying $\pi_0$ to each space of arrows, one gets an ordinary category $\mbox{ho}\left(\C\right)$. A continuous functor $F:\C  \to \sD$ between such topologically enriched categories is said to be an \emph{equivalence} if it induces an equivalence of categories $$\mbox{ho}\left(\C\right) \to \mbox{ho}\left(\sD\right)$$ and if the induced maps on spaces of arrows $$\Hom_\C\left(X,Y\right) \to \Hom_\sD\left(FX,FY\right)$$ is a weak homotopy equivalence for each $X$ and $Y$.

In particular, since associated to any two topological spaces $X$ and $Y$ one has the space of maps $\mathbf{Map}\left(X,Y\right),$\footnote{at least when restricting to compactly generated spaces.} the collection of all $\infty$-groupoids assembles into an $\icat$ $\iGpd$. This is often called the $\infty$-category of \emph{spaces} since the objects may equivalently be thought of as homotopy types. For any $n,$ there is also the subcategory $\nGpd$ of $n$-groupoids (or $n$-types). In particular $\iGpd$ contains the ordinary category of groupoids $\Gpd$ as a full subcategory.

To any Quillen model category $\mathcal{M},$ there is an associated $\icat.$ This is easiest to describe in the case of a simplicial model category, in which case, the $\icat$ associated to $\mathcal{M}$ is the topologically enriched category whose objects are the objects of $\mathcal{M}$ which are both fibrant and cofibrant, and whose spaces of arrows are obtained by geometrically realizing each simplicial mapping space $\Hom_\Delta\left(X,Y\right).$ When two model categories $\mathcal{M}$ and $\mathcal{N}$ are Quillen equivalent, then their associated $\i$-categories are also equivalent. In particular, the $\icat$ associated to the standard Quillen model structure on simplicial sets is another presentation for the $\icat$ $\iGpd$, thanks to the celebrated Quillen equivalence:

\begin{equation}\label{eq:Quillen}
\xymatrix@C=2.5cm{\Set^{\Delta^{op}} \ar@<+0.5ex>[r]^-{|\mspace{3mu}\cdot\mspace{3mu}|} & \mathbf{Top} \ar@<+0.5ex>[l]^-{\operatorname{Sing}}}
\end{equation}
between simplicial sets and topological spaces.

\begin{rmk}
Since simplicial sets are another model for $\i$-groupoids, another model for $\i$-categories is categories enriched in simplicial sets. Using this model, the $\icat$ associated to a simplicial model category $\mathcal{M}$ can be expressed more naturally: it is the category enriched in simplicial sets whose objects are the objects of $\mathcal{M}$ which are both fibrant and cofibrant, and whose spaces of arrows are given by the simplicial mapping spaces $\Hom_\Delta\left(X,Y\right).$
\end{rmk}

There is another powerful framework for modeling  $\i$-categories, and this is the framework of quasicategories (aka inner Kan complexes). In this model, an $\icat$ is a simplicial set $\C$ satisfying certain properties. One of the main advantages of this framework is that a large portion of \cite{htt} is dedicated to working out the higher categorical analogues of much of classical category theory in this setting. For instance, there is a good notion of limit and colimit in this context. If a diagram in the $\icat$ of $\i$-groupoids can be realized as a diagram in topological spaces, then its colimit is represented by the homotopy colimit of the corresponding diagram of spaces, and similarly if it can be realized as a diagram of simplicial sets instead by (\ref{eq:Quillen}). This is more generally true for diagrams in $\i$-categories associated to model categories.

\subsection{Higher Presheaves}
If $\C$ is an $\icat$ one can consider the $\icat$ of contravariant functors from $\C$ to the $\icat$ of $\i$-groupoids $\iGpd$, $$\Fun\left(\C^{op},\iGpd\right)=:\Pshi\left(\C\right),$$ called the $\icat$ of $\i$-presheaves on $\C$. This plays the same role as the category of presheaves does for an ordinary category. In particular, there is a Yoneda embedding: Given an object $X$ of $\C,$ $$y\left(X\right)=\Hom_\C\left(\bl,X\right)$$ defines an $\i$-presheaf on $\C.$ The assignment $$X \mapsto y\left(X\right)$$ extends to a fully faithful functor $$y:\C \hookrightarrow \Pshi\left(\C\right).$$

\begin{lem}\textbf{The $\i$-Yoneda Lemma} (Lemma 5.5.2.1 of \cite{htt})
If $\C$ is an $\icat$ with $X$ an object of $\C,$ and $$F:\C^{op} \to \iGpd$$ is an $\i$-presheaf on $\C,$ then there is a canonical equivalence of $\i$-groupoids $$\Hom\left(y\left(X\right),F\right) \simeq F\left(X\right).$$
\end{lem}

Since the Yoneda embedding is fully faithful, we will often abuse notation and use the same notation $C$ for both the object $C$ in $\C$ and the object $y\left(C\right)$ in $\Pshi\left(\C\right).$

An immediate consequence of the Yoneda lemma is the following basic fact:

Let $F$ be an $\i$-presheaf on $\C,$ and let $\C/F$ denote the full subcategory of the slice $\icat,$ $\Pshi\left(\C\right)/F$ consisting of the objects of the form $y\left(X\right) \to F$, with $X$ an object of $\C.$ There is a canonical projection functor $$\pi_F:\C/F \to \C.$$

\begin{prop}
$F$ is the colimit of the composite $$\C/F \stackrel{\pi_F}{\longlongrightarrow} \C \stackrel{y}{\hookrightarrow} \Pshi\left(\C\right).$$ We often express this by the following informal notation:
\begin{equation}\label{eq:univcolim}
F \simeq \underset{y\left(X\right) \to F}\colim \mspace{2mu} y\left(X\right).
\end{equation}
\end{prop}

If $\C$ is a small $\icat,$ the $\icat$ $\Pshi\left(\C\right)$ of $\i$-presheaves satisfies an analogous universal property to the category of presheaves on a small category:

Let $\sD$ be a cocomplete $\icat$ (meaning it has all small colimits). Suppose that $$\Theta:\Pshi\left(\C\right) \to \sD$$ is a colimit preserving functor. Then by equation (\ref{eq:univcolim}), it follows that 
\begin{equation}\label{eq:lan}
\Theta\left(F\right) \simeq \underset{y\left(X\right) \to F}\colim \mspace{2mu} \Theta\left(y\left(X\right)\right).
\end{equation}
Hence, $\Theta$ is determined by the composite $\theta:=\Theta \circ y:\C \to \sD.$ Conversely, given an arbitrary functor $$\theta:\C \to \sD$$ one can \emph{extend} $\theta$ to a colimit preserving functor $$\Theta:\Pshi\left(\C\right) \to \sD,$$ which is unique with the property that $\Theta \circ y \simeq \theta,$ and this can be done by the formula
$$\Theta\left(F\right) = \underset{y\left(X\right) \to F}\colim \mspace{2mu} \theta\left(X\right).$$ The functor $\Theta$ is called the \textbf{left Kan extension} of $\theta$ along $y$ is denoted $$\Theta=\Lan_y \theta.$$ In fact, we have the following theorem:

\begin{thm}(Theorem 5.1.5.6 of \cite{htt})\label{thm:5.1.5.6}
Let $\C$ be a small $\icat,$ then composition with the Yoneda embedding $$y:\C \hookrightarrow \Pshi\left(\C\right)$$ induces an equivalence of $\i$-categories
$$\Fun^L\left(\Pshi\left(\C\right),\sD\right) \to \Fun\left(\C,\sD\right),$$
between the $\icat$ of colimit preserving functors from $\Pshi\left(\C\right)$ to $\sD$ and the $\icat$ of arbitrary functors from $\C$ to $\sD,$ with inverse given by the functor assigning to each functor $\theta:\C \to \sD$ its left Kan extension $\Lan_y \theta$.
\end{thm}

We use the following observation a few times in this paper:

\begin{rmk}
Given a colimit preserving functor $\Theta=\Lan_y \theta:\Pshi\left(\C\right) \to \sD,$ it automatically has a right adjoint $R$ (by Corollary 5.5.2.9 of \cite{htt}). By the Yoneda lemma, we can determine $R$ explicitly. For $D$ an object of $\sD$ and $C$ an object of $\C,$ we have
\begin{eqnarray*}
R\left(D\right)\left(C\right) &\simeq& \Hom_\C\left(y\left(C\right),R\left(D\right)\right)\\
&\simeq& \Hom_\sD\left(\Theta\left(y\left(C\right)\right),D\right)\\
&\simeq & \Hom_\sD\left(\theta\left(C\right),D\right).
\end{eqnarray*}
\end{rmk}

Finally, we have the following proposition:

\begin{prop}\label{prop:colimlan}
If $t:\C \to \iGpd$ is the constant functor with value the contractible $\i$-groupoid $*,$ then $$\Lan_y\left(t\right)=\colim\left(\mspace{3mu} \cdot \mspace{3mu}\right),$$ i.e. $$\Lan_y\left(t\right):\Pshi\left(\C\right) \to \iGpd$$ sends each presheaf $F:\C^{op} \to \iGpd$ to its colimit.
\end{prop}

\begin{proof}
Denote by $R$ the right adjoint to $\Lan_y\left(t\right).$ Then by the preceding remark, we have that for $C$ an object of $\C$ and $X$ an $\i$-groupoid that
$$R\left(X\right)\left(C\right)\simeq \Hom\left(*,X\right) \simeq X,$$ so we can identify $R\left(X\right)$ with the constant functor $$\Delta_X:\C^{op} \to \iGpd$$ with value $X.$ Since $\Lan_y\left(t\right) \dashv R,$ we have that for $F$ an $\i$-presheaf and $X$ an $\i$-groupoid,
$$\Hom\left(\Lan_y\left(t\right)\left(F\right),X\right) \simeq \Hom\left(F,\Delta_X\right),$$ i.e. maps out of $\Lan_y\left(t\right)\left(F\right)$ are the same as cocones for $F$ with vertex $X$ $$F \Rightarrow \Delta_X,$$ which by definition means that $\Lan_y\left(t\right)\left(F\right)= \colim F.$
\end{proof}

\subsection{Stacks}\label{sec:stacks}

\begin{dfn}\label{dfn:isheave}
Suppose that $\X:\Mfd^{op} \to \iGpd$ is an $\i$-presheaf. We say that $\X$ is an $\i$-stack (or $\i$-sheaf) if for any manifold $M$ and any open cover $\left(U_\alpha \hookrightarrow M\right),$ the canonical map

\begin{equation}\label{eq:descent} 
\mathscr{X}\left(M\right) \to \varprojlim \left[\prod\limits_\alpha \X\left(U_\alpha\right) \rrarrow \prod\limits_{\alpha,\beta} \X\left(U_\alpha \cap U_\beta\right) \rrrarrow \prod\limits_{\alpha,\beta,\gamma} \X\left(U_\alpha \cap U_\beta\cap U_\gamma\right)\ldots\right]
\end{equation}
is an equivalence of $\i$-groupoids, where the above diagram is the entire cosimplicial diagram, with terms involving $n$-fold intersections of elements of the cover for all $n$. Denote the full subcategory of $\Pshi\left(\Mfd\right)$ on the $\i$-stacks by $\Shi\left(\Mfd\right).$ The canonical inclusion $$\Shi\left(\Mfd\right) \hookrightarrow \Pshi\left(\Mfd\right)$$ admits a left adjoint $a$ called the \textbf{$\i$-stackification} functor, and $a$ preserve finite limits.
\end{dfn}

\begin{prop}\label{prop:ithesame}
One can define $\i$-stacks over $\widetilde{\Mfd}$ in a completely analogous way, and there is a canonical equivalence of $\i$-categories $\Shi\left(\widetilde{\Mfd}\right) \simeq \Shi\left(\Mfd\right).$
\end{prop}

\begin{proof}
By the Comparison Lemma of \cite{SGA} III, one has a canonical equivalence between categories of sheaves of sets $$\Sh\left(\Mfd\right) \simeq \Sh\left(\widetilde{\Mfd}\right).$$ Since every manifold is locally of finite covering dimension, it follows readily that both of the $\i$-topoi in question are hypercomplete. The result now follows by noticing that the Comparison Lemma induces an equivalence $$\Sh\left(\Mfd\right)^{\Delta^{op}} \simeq \Sh\left(\widetilde{\Mfd}\right)^{\Delta^{op}}$$ between simplicial sheaves, and invoking Theorem 5 of \cite{Jardine} and Proposition 6.5.2.14 of \cite{htt}.
\end{proof}

\begin{rmk}
Any $\i$-presheaf can be modeled by an ordinary functor $$X:\Mfd^{op} \to \Set^{\Delta^{op}}$$ from (the opposite of) the category of manifolds to the category of simplicial sets, such that each simplicial set $X\left(M\right)$ is a Kan complex. Such an $X$ models an $\i$-sheaf if (\ref{eq:descent}) is a weak equivalence of simplicial sets, where the limit is taken to be a homotopy limit.
\end{rmk}

\begin{dfn}
Denote by $\Psh_1\left(\Mfd\right)$ the bicategory of (possibly weak) $2$-functors $$\X:\Mfd^{op} \to \Gpd$$ from (the opposite of) the category $\Mfd$ into the bicategory of (essentially small) groupoids. This is a full subcategory of the $\icat$ $\Pshi\left(\Mfd\right).$ Such a functor $\X$ is a \textbf{stack} if the composite $$\Mfd^{op} \stackrel{\X}{\longlongrightarrow} \Gpd \hookrightarrow \iGpd$$ is an $\i$-stack. This is equivalent to demanding that for any manifold $M$ and any open cover $\left(U_\alpha \hookrightarrow M\right),$ the canonical map
\begin{equation}
\mathscr{X}\left(M\right) \to \varprojlim \left[\prod\limits_\alpha \X\left(U_\alpha\right) \rrarrow \prod\limits_{\alpha,\beta} \X\left(U_\alpha \cap U_\beta\right) \rrrarrow \prod\limits_{\alpha,\beta,\gamma} \X\left(U_\alpha \cap U_\beta\cap U_\gamma\right)\right]
\end{equation}
is an equivalence of groupoids, where the limit is a bicategorical limit, or homotopy limit, in the bicategory $\Gpd$ of groupoids. Notice that unlike (\ref{eq:descent}), the above diagram stops at $3$-fold intersections. One could use the entire untruncated cosimplicial diagram used in (\ref{eq:descent}), however, in the case of groupoids, this would result in the same limit. Denote the full subcategory of $\Psh_1\left(\Mfd\right)$ on the stacks by $\St\left(\Mfd\right)$. The canonical inclusion $$\St\left(\Mfd\right) \hookrightarrow \Psh_1\left(\Mfd\right)$$ admits a left adjoint $a$ called the \textbf{stackification} functor, and $a$ preserve finite limits.
\end{dfn}

\begin{rmk}\label{rmk:compy}
The canonical equivalence $\Shi\left(\widetilde{\Mfd}\right) \simeq \Shi\left(\Mfd\right)$ restricts to an equivalence $\St\left(\widetilde{\Mfd}\right) \simeq \St\left(\Mfd\right)$ as any equivalence of $\i$-categories induces an equivalence on the associated subcategories of $1$-truncated objects.
\end{rmk}

\begin{rmk}
There is no real danger in denoting both the stackification and $\i$-stackification functor by $a$ as by Proposition 5.5.6.16 of \cite{htt}, the following diagram commutes:

$$\xymatrix{\Psh_1\left(\Mfd\right) \ar[r]^-{a} \ar@{^{(}->}[d] & \St\left(\Mfd\right) \ar@{^{(}->}[d]\\
\Pshi\left(\Mfd\right) \ar[r]^{a} & \Shi\left(\Mfd\right).}$$
\end{rmk}

Let $\G$ be a Lie groupoid. Then $\G$ determines a weak presheaf of groupoids on $\Mfd$ by the rule

\begin{equation*}
M \mapsto \Hom_{\LieGpd}\left(M^{id},\G\right),
\end{equation*}
This defines a $2$-functor $\tilde y: \LieGpd \to \Psh_1\left(\Mfd\right)$ and we have the obvious commutative diagram

$$\xymatrix{\Mfd  \ar[d]_{\left(\mspace{2mu} \cdot \mspace{2mu}\right)^{(id)}} \ar[r]^{y} & \Psh\left(\Mfd\right) \ar^{\left(\mspace{2mu} \cdot \mspace{2mu}\right)^{(id)}}[d]\\
\LieGpd \ar_{\tilde y}[r] & \Psh_1\left(\Mfd\right),}$$
where $y$ denotes the Yoneda embedding. 

\begin{rmk}
An $\i$-presheaf $\X:\Mfd^{op} \to \iGpd$ is an \textbf{$\i$-stack} (or $\i$-sheaf) if and only if for every manifold $M$ and every open cover $\mathcal{U}=\left(U_\alpha \hookrightarrow M\right),$ the canonical map
$$\Hom\left(y\left(M\right),\X\right) \to \Hom\left(\tilde y\left(M_\mathcal{U}\right),\X\right)$$ is an equivalence of $\i$-groupoids. Similarly for stacks of groupoids. That is to say, by the Yoneda lemma one has
$$\Hom\left(y\left(M\right),\X\right) \simeq \X\left(M\right)$$ and by Remark \ref{rmk:Gcolim}, there is a canonical equivalence of $\i$-groupoids
$$\Hom\left(\tilde y\left(M_\mathcal{U}\right),\X\right) \simeq \varprojlim \left[\prod\limits_\alpha \X\left(U_\alpha\right) \rrarrow \prod\limits_{\alpha,\beta} \X\left(U_\alpha \cap U_\beta\right) \rrrarrow \prod\limits_{\alpha,\beta,\gamma} \X\left(U_\alpha \cap U_\beta\cap U_\gamma\right)\ldots\right].$$
\end{rmk}

\section{Principal Bundles for Lie Groupoids}

\begin{dfn}\label{def:Gmfd}
Given a Lie groupoid $\G$, a (left) \textbf{$\G$-manifold} is a (possibly non-Hausdorff or $2^{nd}$-countable) manifold $E$ equipped with a \textbf{moment map} $\mu:E \to \G_0$ and an \textbf{action map} $\rho:\G_1 \times_{\G_0} E \to E$,
where
$$\xymatrix{\G_1 \times_{\G_0} E \ar[r]  \ar[d] & E \ar^-{\mu}[d] \\
\G_1 \ar^-{s}[r] & \G_0\\}$$
is the fibered product, such that the following conditions hold:

\begin{itemize}
\item[i)] $\left(gh\right) \cdot e = g \cdot \left(h \cdot e\right)$ whenever $e$ is an element of $E$ and $g$ and $h$ elements of $\G_1$ with domains such that the composition makes sense
\item[ii)] $1_{\mu\left(e\right)} \cdot e =e$ for all $e \in E$
\item[iii)] $\mu\left(g \cdot e\right) = t \left(g\right)$ for all $g \in \G_1$ and $e \in E$.
\end{itemize}
\end{dfn}

\begin{dfn}\label{dfn:Gbundle}
A (left) \textbf{$\G$-bundle} over a manifold $M$ is a $\G$-manifold $P$ equipped with a $\G$-invariant \textbf{projection map} $$\pi:P \to M$$ which is a surjective submersion. Such a $\G$-bundle is \textbf{principal} if the induced map, $$\G_1 \times_{\G_0} P \to P \times_{M} P$$ is a diffeomorphism. We usually use the notation
$$\xymatrix @R=2pc @C=0.15pc {\G_1  \ar@<+.7ex>[d] \ar@<-.7ex>[d] & \acts & P \ar_{\mu}[lld] \ar^{\pi}[d] \\
\G_0 && M}$$ to denote such a principal $\G$-bundle.
\end{dfn}

\begin{rmk}
When $\G=G$ is a Lie group, then the above definition agrees with the classical definition of a principal $G$-bundle over $M$.
\end{rmk}

\begin{dfn}\label{dfn:bibundle}
Let $\G$ and $\H$ be two Lie groupoids. A \textbf{principal $\G$-bundle over $\H$} is a principal $\G$-bundle

$$\xymatrix @R=2pc @C=0.15pc {\G_1  \ar@<+.7ex>[d] \ar@<-.7ex>[d] & \acts & P \ar_{\mu}[lld] \ar^{\nu}[d] \\
\G_0 && \H_0}$$
over $\H_0$, such that $P$ also has the structure of a right $\H$-bundle with moment map $\nu$, with the $\G$ and $\H$ actions commuting in the obvious sense. We typically denote such a bundle by

$$\xymatrix @R=2pc @C=0.15pc {\G_1  \ar@<+.7ex>[d] \ar@<-.7ex>[d] & \acts & P \ar_{\mu}[lld] \ar^{\nu}[rrd] & \acted & \H_1 \ar@<+.7ex>[d] \ar@<-.7ex>[d]\\
\G_0 & & & &\H_0.}$$
\end{dfn}

Given a morphism of stacks $\left[\H\right] \to \left[\G\right],$ there is a pullback diagram
\begin{equation}\label{eq:pbprin}
\xymatrix{P \ar[rr]^-{\mu} \ar[d]_-{\nu} & & \G_0 \ar[d]\\
\H_0 \ar[r] & \left[\H\right] \ar[r] & \left[\G\right]}
\end{equation}
with $P$ a manifold, and $P$ inherits the structure of a principal $\G$-bundle over $\H.$ Conversely, given a principal $\G$-bundle $P$ over $\H,$ it canonically gives rise to a morphism of stacks $$p:\left[\H\right] \to \left[\G\right]$$ such that $P$ can be obtained by the pullback diagram (\ref{eq:pbprin}). In fact, there is a way of ``composing'' a principal $\cK$ bundle over $\G$ with a principal $\H$-bundle over $\G$ to produce a principal $\cK$-bundle over $\H$. This gives rise to the structure of a bicategory whose objects are Lie groupoids, arrows are principal bundles, and $2$-arrows are morphisms of principal bundles. This bicategory is canonically equivalent to the bicategory of differentiable stacks. See \cite{Christian} (Theorem 2.18) and \cite{thesis} (Theorem I.2.4) for more details.

\begin{cor}\label{cor:buncon}
For $M$ a manifold, $\left[\G\right]\left(M\right)$ is canonically equivalent to the groupoid of principal $\G$-bundles over $M.$ Moreover, given a morphism $p:M \to \left[\G\right],$ the principal bundle $P$ corresponding to $p$ is obtained as the pullback diagram
$$\xymatrix{P \ar[r]^-{\mu} \ar[d]_-{\pi} & \G_0 \ar[d]\\
M \ar[r] & \left[\G\right].}$$ In particular, for $G$ a Lie group, $\left[G\rrarrow *\right]$ is the stack of principal $G$-bundles.
\end{cor}

\section{Proof of Theorem \ref{thm:embeql}} \label{sec:monoid}

\begin{lem}\label{lem:slicern}
The slice $\infty$-topos $$\Shi\left(\Emb\right)/ay\left(\RR^n\right)$$ is canonically equivalent to $\Shi\left(\RR^n\right)$ where $y$ denotes the Yoneda embedding.
\end{lem}

\begin{proof}
Denote by $\mathcal{O}p\left(\R^n\right)$ the poset of open subsets of $\R^n$ and denote by $\mathcal{O}p\left(\R^n\right)'$ the full subcategory on those open subsets that are abstractly diffeomorphic to $\R^n.$ We claim that the canonical inclusion $$l:\mathcal{O}p\left(\R^n\right)' \hookrightarrow \mathcal{O}p\left(\R^n\right)$$ induces an equivalence $$\Shi\left(\R^n\right):=\Shi\left(\mathcal{O}p\left(\R^n\right)\right) \to \Shi\left(\mathcal{O}p\left(\R^n\right)'\right),$$ where we endow $\mathcal{O}p\left(\R^n\right)'$ with the Grothendieck topology generated by good open covers. To see this, consider the left Kan extension
$$\Lan_y\left( ay\circ l\right):\Pshi\left(\mathcal{O}p\left(\R^n\right)'\right) \to \Shi\left(\mathcal{O}p\left(\R^n\right)\right).$$ It has a right adjoint $l^*$ given by composition with $l$. Notice that $l^*$ lands in sheaves since $l$ preserves covers. Hence, there is an induced adjunction
$$\xymatrix{\Shi\left(\mathcal{O}p\left(\R^n\right)'\right) \ar@<-0.5ex>[r]_-{l_!} & \Shi\left(\mathcal{O}p\left(\R^n\right)\right) \ar@<-0.5ex>[l]_-{l^*}}$$ with $l_! \dashv l^*$. We claim that $l^*$ also has a right adjoint $l_*$. By the Yoneda lemma, if there existed such an adjoint, it would have to satisfy
\begin{eqnarray*}
 l_*F\left(U\right) &\simeq&  \Hom\left(y\left(U\right),l_*F\right)\\
&\simeq& \Hom\left(l^*y\left(U\right),F\right),
\end{eqnarray*}
so such an adjoint exists if and only if each such $l_*F$ is an $\i$-sheaf. To this end, it suffices to show that if $\left(U_i \le U\right)$ is a cover of $U$ then
$$l^*y\left(U\right) \simeq \colim \left[ \ldots \coprod l^*y\left(U_{ijk}\right) \rrrarrow \coprod l^*y\left(U_{ij}\right) \rrarrow \coprod l^*y\left(U_i\right)\right].$$ Since both $y$ and $l^*$ preserve finite limits, we have that the simplicial diagram
$$ \ldots \coprod l^*y\left(U_{ijk}\right) \rrrarrow \coprod l^*y\left(U_{ij}\right) \rrarrow \coprod l^*y\left(U_i\right)$$ is the \v{C}ech nerve of the morphism $$\coprod l^*y\left(U_i\right) \to l^*y\left(U\right),$$ and hence the above colimit can be identified with a subobject of $l^*y\left(U\right)$- which itself is a $\left(-1\right)$-truncated object. We can therefore compute this colimit in sheaves of sets. This amounts to computing it first in presheaves, and then sheafifying. Denote this presheaf obtained in this way by $G_U.$ For $V \in \mathcal{O}p\left(\R^n\right)',$ we have by general considerations that $G_U\left(V\right)$ is the subset of $\Hom_{\mathcal{O}p\left(\R^n\right)}\left(V,U\right)$ on those morphisms $f:V \to U$ which factor through some $U_i.$ Since we are in a poset, there is only a morphism $V \to U$ if $V \subseteq U$ and if so it is unique, so we have that $G_U\left(V\right)$ is the singleton set if there exists an $i$ such that $V \subseteq U_i.$ The sheafification of $G_U$ with respect to good open covers is easily seen to be $l^*y\left(U\right)$. It follows that $l_*F$ is always a sheaf, and we deduce that $l^*$ has a right adjoint, hence preserves all colimits.

Now, if $V$ is in $\mathcal{O}p\left(\R^n\right)'$, then we have a canonical equivalence $$l^*l_!y\left(V\right) \simeq y\left(V\right).$$ Now suppose that $U$ is in $\mathcal{O}p\left(\R^n\right).$ Then we can choose a good open cover $\left(U_i \le U\right)$ and write 
$$y\left(U\right) = \colim \left[ \ldots \coprod y\left(U_{ijk}\right) \rrrarrow \coprod y\left(U_{ij}\right) \rrarrow \coprod y\left(U_i\right)\right],$$ and since both $l^*$ and $l_!$ preserve colimits, together with the fact that each $U_i$ are in $\mathcal{O}p\left(\R^n\right)'$ we have that
\begin{eqnarray*}
l_!l^*\left(y\left(U\right)\right) &\simeq& \colim \left[ \ldots \coprod l_!l^*y\left(U_{ijk}\right) \rrrarrow \coprod l_!l^*y\left(U_{ij}\right) \rrarrow \coprod l_!l^*y\left(U_i\right)\right]\\
&\simeq& \colim \left[ \ldots \coprod y\left(U_{ijk}\right) \rrrarrow \coprod y\left(U_{ij}\right) \rrarrow \coprod y\left(U_i\right)\right]\\
&\simeq& y\left(U\right).
\end{eqnarray*}
It follows that the components of the co-unit $$l_!l^* \to id$$ and the unit $$id \to l^*l_!$$ along representables are equivalences, and hence by Theorem \ref{thm:5.1.5.6}, the co-unit and unit are equivalences, and hence $l_! \dashv l^*$ is an adjoint equivalence.

By Proposition 2.2.1 of \cite{higherme}, we have that $$\Shi\left(\Emb\right)/ay\left(\RR^n\right)\simeq \Shi\left(\Emb/\RR^n\right).$$ There is a canonical fully faithful functor $$\Emb/\RR^n\hookrightarrow \mathcal{O}p\left(\R^n\right)' $$ sending an embedding $\varphi:\R^n \hookrightarrow \R^n$ to the open subset $\varphi\left(\R^n\right).$ If $U \subseteq \R^n$ is an object of $\mathcal{O}p\left(\R^n\right)'$, then its canonical inclusion $U \hookrightarrow \R^n$ is an object of $\Emb/\RR^n$, hence the above functor is also surjective on objects, and hence an equivalence. Therefore we have that $$\Shi\left(\Emb/\RR^n\right)\simeq \Shi\left(\mathcal{O}p\left(\R^n\right)'\right) \simeq \Shi\left(\R^n\right).$$
\end{proof}

\begin{lem}\label{lem:locfincov}
The $\infty$-topos $\Shi\left(\Emb\right)$ is locally of homotopy dimension $\le n$ in the sense of Definition 7.2.1.5 of \cite{htt}.
\end{lem}

\begin{proof}
The image of the composite $$\Emb \stackrel{y}{\hookrightarrow} \Pshi\left(\Emb\right) \stackrel{a}{\longrightarrow} \Shi\left(\Emb\right)$$ is strongly generating, so we conclude in particular that $\Shi\left(\Emb\right)$ is generated under colimits by $ay\left(\RR^n\right).$ It therefore suffices to show that the slice $\infty$-topos $$\Shi\left(\Emb\right)/ay\left(\RR^n\right)$$ is of homotopy dimension $\le n$. However, by Lemma \ref{lem:slicern}, one has $$\Shi\left(\Emb\right)/ay\left(\RR^n\right) \simeq \Shi\left(\R^n\right)$$ and since $\R^n$ is a paracompact space of topological covering dimension $\le n,$ we are done by Theorem 7.2.3.6 of \cite{htt}.
\end{proof}

\begin{thm}\label{thm:embeql}
The canonical functor $$\alpha^*:\Pshi\left(\nMfd^{\et}\right) \to \Pshi\left(\Emb\right)$$ restricts to an equivalence of $\i$-categories
$$\alpha^*:\Shi\left(\nMfd^{\et}\right) \to \Shi\left(\Emb\right).$$
\end{thm}

\begin{proof}
First observe that $\alpha$ clearly preserves covers. Hence $\alpha^*$ restricts to a functor between the respective subcategories of $\i$-sheaves. Second, by the Comparison Lemma of \cite{pres} (p. 151), which is a slight generalization of the lemma by the same name in \cite{SGA}, we can deduce that the induced functor between sheaves with values in $\Set$ $$\alpha^*:\Sh\left(\nMfd^{\et}\right) \to \Sh\left(\Emb\right)$$ is an equivalence of categories. The only non-trivial property of the $1)$-$5)$ listed in loc. cit. is $2)$- the locally full property. To verify this property, we must show that given any local diffeomorphism $$f:\RR^n \to \RR^n,$$ there exists a cover by self-embeddings $$\varphi_i:\RR^n \to \RR^n$$ such that each composite $f \circ \varphi_i$ is an embedding. Clearly, given a local diffeomorphism $f$ as above, there exists an open cover $$\left(U_\alpha \hookrightarrow \RR^n\right),$$ such that for all $\alpha,$ $f|_{U_\alpha}$ is an embedding. Refining the open cover if necessary, we may assume that the $U_\alpha$ is a coordinate patch diffeomorphic to $\RR^n,$ hence establishing our claim.

Now by Remark 6.1.2 of \cite{higherme} together with Theorem 5.3.6 of the same article, it follows that $\Shi\left(\nMfd^{\et}\right)$ is hypercomplete. By Lemma \ref{lem:locfincov} together with Corollary 7.2.1.12 of \cite{htt}, it follows that $\Shi\left(\Emb\right)$ is as well. The result now follows since the above paragraph implies an equivalence of categories $$\Sh\left(\nMfd^{\et}\right)^{\Delta^{op}} \simeq \Sh\left(\Emb\right)^{\Delta^{op}}$$ between simplicial sheaves, so we can now invoke Theorem 5 of \cite{Jardine} and Proposition 6.5.2.14 of \cite{htt}.
\end{proof}

\bibliography{meanet}
\bibliographystyle{hplain}

\end{document}